\documentclass[10pt,a4paper]{article}

\usepackage{tikz}
\usepackage[width=14.00cm, height=25.00cm]{geometry}
\usepackage[T1]{fontenc}
\usepackage[utf8]{inputenc}
\usepackage{amsmath,amssymb,amsfonts,amsthm,amscd, bm, mathtools}
\usepackage{enumerate}
\usepackage{todonotes}
\usepackage{hyperref}
\usepackage{float}
\usepackage{comment}

\usetikzlibrary{positioning}
\usetikzlibrary{decorations.pathmorphing}
\usetikzlibrary{decorations.text}
\usepackage{indentfirst}
\newcommand{\bref}[1]{\textbf{\ref{#1}}}

\theoremstyle{plain}
\newtheorem{theorem}{Theorem}[section]
\newtheorem{proposition}{Proposition}[section]
\newtheorem{lemma}{Lemma}[section]
\newtheorem{corollary}{Corollary}[section]
\newtheorem{conjecture}{Conjecture}[section]
\newtheorem{claim}{Claim}[section]
\newtheorem{question}{Question}[section]

\newtheorem*{theorem*}{Theorem}
\newtheorem*{corollary*}{Corollary}

\newtheoremstyle{theoremstar}
{} 
{} 
{\itshape} 
{} 
{\bfseries} 
{.} 
{.5em} 
{\thmname{#1}\thmnote{ \bfseries #3}}
\makeatother

\theoremstyle{theoremstar}
\newtheorem*{theoremstar}{Theorem}

\newtheoremstyle{corollarystar}
{} 
{} 
{\itshape} 
{} 
{\bfseries} 
{.} 
{.5em} 
{\thmname{#1}\thmnote{ \bfseries #3}}
\makeatother

\theoremstyle{corollarystar}
\newtheorem*{corollarystar}{Corollary}

\theoremstyle{definition}
\newtheorem{definition}{Definition}[section]
\newtheorem{example}{Example}[section]

\theoremstyle{remark}

\title{Kriesell's conjecture for infinite graphs}
\author{Leandro Aurichi, Paulo Magalhães Júnior and Rodrigo Monteiro}
\newcommand{\Addresses}{{
		\bigskip
		\footnotesize
		
		L.~Aurichi, Instituto de Ci\^encias Matem\'aticas e de Computa\c c\~ao, Universidade de S\~ao Paulo\\
		Avenida Trabalhador s\~ao-carlense, 400, S\~ao Carlos, SP, 13566-590, Brazil\par\nopagebreak
		\textit{E-mail}: \texttt{aurichi@icmc.usp.br}
		
		\medskip
		P.~Magalhães Jr, Instituto Federal do Rio Grande do Norte\\
		Rua Manoel Lopes Filho 773, Currais Novos, RN, 59380-000 \par\nopagebreak
		\textit{E-mail}: \texttt{paulo.magalhaes@ifrn.edu.br}
		
		\medskip
		R.~Monteiro, Instituto de Ci\^encias Matem\'aticas e de Computa\c c\~ao, Universidade de S\~ao Paulo\\
		Avenida Trabalhador s\~ao-carlense, 400, S\~ao Carlos, SP, 13566-590, Brazil\par\nopagebreak
		\textit{E-mail}: \texttt{rodrigosm@usp.br}
}}
\date{}

\begin{document}
	
\maketitle
	
	\begin{abstract}
		Let $G$ be a graph and $S\subseteq V(G)$ be a subset of vertices. An $S$-Steiner tree $T$ of $G$ is a tree of $G$ which contains $S$ in its vertex set $V(T)$. Kriesell conjectured that for every $2k$-edge-connected subset $S\subseteq V(G)$ in a finite connected graph $G$, there exist $k$ pairwise edge-disjoint $S$-Steiner trees. This conjecture is false for infinite graphs. We present a version of Kriesell's conjecture with topological $S$-Steiner trees for countable finitely edge-separable graphs and a version with $F$-limits of trees for rayless graphs.		
We show that if Kriesell's conjecture holds for finite graphs, then it holds for every connected, rayless and finitely edge-separable graph. We also show that every $2k$-edge-connected rayless and finitely edge-separable graph contains $k$ pairwise edge-disjoint spanning trees.
	\end{abstract}



	\section{Introduction}
	
	In this paper, all graphs are assumed to be simple, that is, they contain neither loops nor multiple edges. We adopt the terminology and notation of~\cite{Diestel2025}.
Let $G$ be a graph. We say that a subset $S \subseteq V(G)$ is \emph{$k$-edge-connected} if, for every pair of vertices $u,v \in S$, there exist $k$ pairwise edge-disjoint $u$–$v$ paths in $G$. In the case $S = V(G)$, the graph $G$ is said to be \emph{$k$-edge-connected}. 

The following well-known and important theorem, due to Nash-Williams and Tutte, relates even edge-connectivity to the existence of edge-disjoint spanning trees:

	\begin{theorem}[\cite{nash1961edge,tutte1961problem}]
		\label{nash}
		Every $2k$-edge-connected finite graph contains $k$ pairwise edge-disjoint spanning trees.
	\end{theorem}

Theorem~\ref{nash} does not extend to infinite graphs.  For every natural number $k\in\mathbb{N}$ greater than two, one can construct a $2k$-edge-connected countable graph that does not contain two edge-disjoint spanning trees (see~\cite{AharoniThomassen1989}). However, it remains valid for locally finite infinite graphs if one considers topological analogues of spanning trees, namely, topological spanning trees.

\begin{theorem}[{\cite[Chapter 8]{Diestel2025}}]
	\label{teo}
	Every $2k$-edge-connected locally finite graph contains $k$ pairwise edge-disjoint topological spanning trees.
\end{theorem}

The notion of a topological tree arises naturally when one considers a graph together with its ends (or edge-ends) and views it as a topological space.

\begin{definition}
	Let $G$ be a graph. A 1-way infinite path is called a \emph{ray} and its infinite connected subgraphs are its \emph{tails}. We denote by $\mathcal{R}(G)$ the set of rays of a graph $G$.
	Two rays $r,s\in \mathcal{R}(G)$ in a graph $G$ are \emph{edge-equivalent} if no finite set of edges separates them, that is, for every finite set of edges $X\subseteq E(G)$ the rays $r$ and $s$ have a tail in the same connected component of $G-X$. Each equivalence class $[r]_E$ is called an edge-end of $G$. We denote by $\Omega_E(G)$ the set of edge-ends of $G$.
\end{definition}

Edge-ends were first introduced by Hahn, Laviolette and \v{S}ir\'a\v{n} in \cite{HahnLavioletteSiran1997} and have recently been studied in \cite{boska2025edgedirectioncompactedgeendspaces,pitz2025metrizationtheoremedgeendspaces, real2025subbasepropertydescribingedgeend}. Let us define a topology on $||G|| = G \cup \Omega_E(G)$. We begin by viewing $G$ itself as a $1$-complex. Every edge is homeomorphic to the interval $[0,1]$. For every vertex $v\in V(G)$, we choose as a basis of open neighborhoods the open stars of radius $\frac{1}{n}$. It remains to define the basic open sets around the edge-ends of $G$. For every finite subset of edges $X \subseteq E(G)$ and every edge-end $[r]_E \in \Omega_E(G)$, there is a unique component $C_E(X,[r]_E)$ of $G- X$ that contains a tail of every ray in $[r]_E$.  Define $$\Omega_E(X,[r]_E) = \{[r'] \in \Omega_E(G) : C_E(X,[r']_E) = C_E(X,[r]_E)\}$$ Given $[r]_E \in \Omega_E(G)$ and a finite set $X\subset E(G)$, we define the open neighborhood around $[r]_E$ to be
$$
\Omega_E(X,[r]_E) \cup C_E(X,[r]_E) \cup E(X,[r]_E)
$$
where $E(X,[r]_E)$ is the union
of half-edges between $C_E(X,[r]_E)$ and $X$. We denote this topology by $\textsc{Etop}(G) $.

Let $G$ be a graph, let $u,v \in V(G)$, and let $[r]_E \in \Omega_E(G)$. We say that the vertex $u$ lies in the edge-end $[r]_E$ if it edge-dominates a ray in $[r]_E$, that is, if for every finite set of edges $X\subseteq E(G)$ there exists a ray in $[r]_E$ with a tail in the same connected component as $u$ in $G-X$. The vertices $u$ and $v$ are infinitely edge-connected if there exist infinitely many edge-disjoint paths between them.

We define the topological space $\textsc{Etop}'(G)$ as the quotient space obtained by identifying edge-inseparable points of $||G||$, that is, points such that every neighborhood of one contains the other and vice versa. This occurs precisely when either one point is an edge-end and the other is a vertex that edge-dominates it, or when both points are infinitely edge-connected vertices. If $G$ is locally finite, then $ \textsc{Etop}'(G)$ and $\textsc{Etop}(G)$ coincide.

\begin{definition}
Let $G$ be a graph. A topological path $P$ is a subspace of $\textsc{Etop}'(G)$ homeomorphic to $[0,1]$, and a topological circle $C$ is a subspace of $\textsc{Etop}'(G)$ homeomorphic to $\mathbb{S}^1$. We say that $C\cap G$ is a topological cycle. A \emph{topological tree} $T$ of $G$ is a path-connected subspace of $\textsc{Etop}'(G)$ that contains no topological circles. If $V(G) = V(T)$, we say that $T$ is a topological spanning tree.

\end{definition}

In Theorem~\ref{nash}, the graph is assumed to be $2k$-edge-connected. In~\cite{Kriesell2003}, Kriesell conjectured  that a stronger result holds for any $2k$-edge-connected vertex set, together with the corresponding generalization of spanning trees, namely, Steiner trees.

	\begin{definition}
			Let $G$ be a graph and $S$ be a subset of vertices of $G$. A (topological) $S$-Steiner tree is a (topological) tree which contains $S$ in its vertex set.
	\end{definition}
		\begin{conjecture}[\cite{Kriesell2003}]
			\label{kriesell}
		Let $G$ be a connected finite graph and let $S \subseteq V(G)$ be a $2k$-edge-connected subset. Then $G$ contains $k$ pairwise edge-disjoint $S$-Steiner trees.
	\end{conjecture}
	
 In~\cite{Kriesell2003}, Kriesell shows that his conjecture holds in the special case where every vertex in $V(G) \setminus S$ has even degree. Frank, Király, and Kriesell~\cite{frank2003} proved that there exist $k$ edge-disjoint $S$-Steiner trees whenever all edge cuts separating $S$ have size greater than $3k$ and $V(G) \setminus S$ is an independent set. Additional cases have been established involving modifications of the edge-connectivity of the set $S$ (see \cite{DeVosMcDonaldPivotto2016,Li2018}). The same counterexample to Theorem~\ref{nash} for infinite graphs also applies to Conjecture~\ref{kriesell}.

The goal of this paper is to propose a version of Kriesell’s conjecture in terms of topological Steiner trees and $F$-limits of trees. We extend this conjecture in stages: first to locally finite graphs, then to countable graphs, and finally to rayless graphs. For locally finite graphs, the conjecture is the following:

\begin{conjecture}
	\label{conj}
	Let $G$ be a connected, locally finite infinite graph and $S\subseteq V(G)$ be a $2k$-edge-connected subset. Then $G$ with $\textsc{Etop}(G)$ contains $k$ pairwise edge-disjoint topological $S$-Steiner trees.
\end{conjecture}

	The locally finite case, which will be established in Section~\ref{loc}, is proved using $F$-limits of trees. 	
	Intuitively, an $F$-limit $H$ of a graph $G$ is a subgraph for which there exists a sequence of subgraphs $\langle H_n\rangle _{n \in \mathbb{N}}$ of $G$ such that a vertex or edge of $G$ belongs to $H$ whenever it belongs to $H_n$ for many $n \in \mathbb{N}$. This technique is presented in this context in \cite{Aurichi2026} and has recently been employed in \cite{guilherme, aurichi2025cyclecoversinfinitebipartite}. We use non-principal ultrafilters to define this limit and make precise the meaning of ``many $n$''.
	Our main theorem of Section \ref{loc} is:

		\begin{theoremstar}[\bref{localfin}]
		 Conjectures~\ref{kriesell} and~\ref{conj} are equivalent.
	\end{theoremstar}

Recall that a graph is finitely edge-separable if every pair of vertices can be separated by a finite set of edges.
In Section~\ref{enum}, we establish the countable case. Moreover, we present a version based on $F$-limits of trees for any countable graph.

\begin{theoremstar}[\bref{teoenum}]
	\label{countable}
	Let $G$ be a connected, finitely edge-separable, countable graph and $S\subseteq V(G)$ be a $2k$-edge-connected subset. If Conjecture \ref{kriesell} holds, then $G$ with $\textsc{Etop}'(G) $ contains $k$ pairwise edge-disjoint topological $S$-Steiner trees. 
\end{theoremstar}

\begin{theoremstar}[\bref{counta}]
	Let $G$ be a connected countable graph and $S\subseteq V(G)$ be a $2k$-edge-connected subset. If Conjecture \ref{kriesell} holds, then $G$ contains $k$ pairwise edge-disjoint $F$-limits of trees which contain $S$. 
\end{theoremstar}

In Section~\ref{rayless}, we restrict our attention to rayless graphs. The tools used in the proof of the rayless case will be bond-faithful decompositions, defined in~\cite{Laviolette2005}. The structure of rayless graphs will enable certain arguments that are key to the proof.

\begin{theoremstar}[\bref{final}]
	Let $G$ be a connected rayless graph and $S\subseteq V(G)$ be a $2k$-edge-connected subset. If Conjecture \ref{kriesell} holds, then $G$ contains $k$ pairwise edge-disjoint $F$-limits of trees which contain $S$.
\end{theoremstar}

The counterexample to Conjecture \ref{kriesell} for infinite graphs is a graph that contains rays. We conclude by presenting a tree-based version for connected, rayless, and finitely edge-separable graphs.

\begin{theoremstar}[\bref{ray}]
	Let $G$ be a connected, finitely edge-separable and rayless graph and $S\subseteq V(G)$ be a $2k$-edge-connected subset. If Conjecture \ref{kriesell} holds, then $G$ contains $k$ pairwise edge-disjoint $S$-Steiner trees. 
\end{theoremstar}
In particular, it is not possible to find a counterexample to Theorem \ref{nash} that is a rayless and finitely edge-separable graph.

\begin{corollarystar}[\bref{final1}]
	Let $G$ be a $2k$-edge-connected, rayless and finitely edge-separable graph. Then $G$ contains $k$ pairwise edge-disjoint spanning trees.
\end{corollarystar}

\section{Locally finite case}
\label{loc}

The main tool we will use to construct topological $S$-Steiner trees in locally finite graphs is the
$F$-limit.

\begin{definition}
	A \emph{filter} $F$ on $\mathbb{N}$ is a nonempty family of subsets of $\mathbb{N}$ such that
	\begin{itemize}
		\item $\varnothing \notin F$;
		\item if $A \in F$ and $A \subseteq B \subseteq \mathbb{N}$, then $B \in F$;
		\item if $A,B \in F$, then $A \cap B \in F$.
	\end{itemize}
	A filter on $\mathbb{N}$ is called an \emph{ultrafilter} if it is not properly contained in any other filter on $\mathbb{N}$. An ultrafilter is \emph{principal} if it contains a singleton; otherwise, it is \emph{non-principal}.
\end{definition}

We will use the following basic properties of an ultrafilter $F$ on $\mathbb{N}$. For every $A\subseteq\mathbb{N}$, exactly one of $A$ and $\mathbb{N}\setminus A$ belongs to $F$. Consequently, if a finite union of sets belongs to $F$, then at least one of those sets belongs to $F$. If $F$ is non-principal, it contains every cofinite subset of $\mathbb{N}$ and no finite subset. In particular, every finite intersection of members of $F$ belongs to $F$ and is infinite.

\begin{definition} Let $F$ be a non-principal ultrafilter over the natural numbers $\mathbb{N}$, $G$ be a graph 
	and $\langle H_n \rangle_{n \in \mathbb{N}}$ be a sequence of subgraphs of $G$. 
	We say that a subgraph $H \subseteq G$ is the \textit{$F$-limit} of $\langle H_n \rangle_{n \in \mathbb{N}}$ 
	if $H = \langle V_H, E_H\rangle$, where
	$$
	V(H) = \{ v \in V(G) : \{ n : v \in V(H_n)\} \in F \}
	$$
	$$
	E(H) = \{ e \in E(G) : \{ n : e \in E(H_n)\} \in F \}
	$$

\end{definition}

The strategy for the locally finite case is to approximate the graph by a sequence of finite connected subgraphs $G_n$, each containing a subset $S_n \subseteq S$ that remains $2k$-edge-connected in $G_n$. Assuming that Conjecture~\ref{kriesell} holds, each $G_n$ admits $k$ pairwise edge-disjoint $S_n$-Steiner trees. We then consider a suitable $F$-limit of these trees. However, an $F$-limit of trees need not be a topological tree (see Figure~\ref{limit}). We overcome this by showing that a topological $S$-Steiner tree can be extracted from an $F$-limit of trees that contains $S$.

\begin{figure}[htbp]
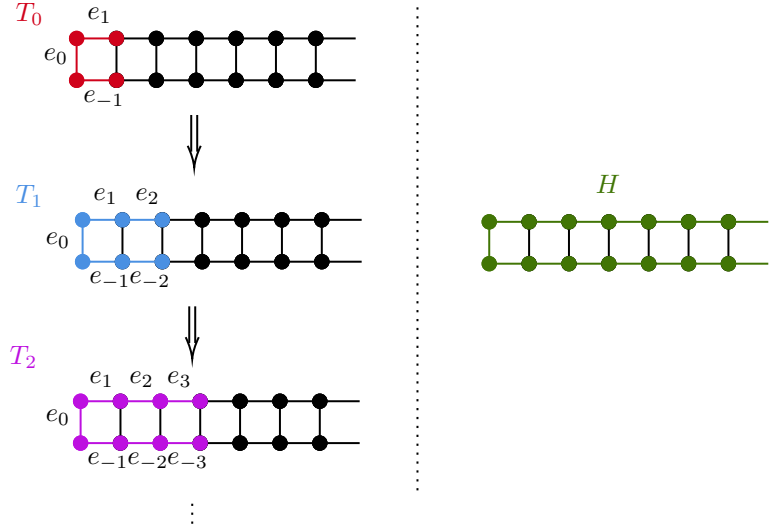

	\centering
	
	\tikzset{every picture/.style={line width=0.75pt}} 
	

		\caption{An $F$-limit of trees which is not a topological tree.}
		\label{limit}
	\end{figure}
	
	\begin{example}
		
Consider the infinite ladder extending to the right, shown in Figure~\ref{limit}. We enumerate the edges of the upper ray by the natural numbers $\mathbb{N}$ in increasing order from left to right, and the edges of the lower ray by the negative integers in decreasing order from left to right. Moreover, let $e_0$ denote the first rung of the ladder, that is, the first edge joining the upper and lower rays. Denote this graph by $H$.

For each natural number $n \in \mathbb{N}$, let $T_n$ be the subgraph with edge set 
$E(T_n)=\{e_j : -(n+1) \le j \le n+1\}$ and vertex set given by the endpoints of these edges. Then $T_n$ is a path, and in particular a finite tree. Fix a non-principal ultrafilter $F$ over $\mathbb{N}$. We claim that the $F$-limit of $\langle T_n \rangle_{n \in \mathbb{N}}$ is $H$.

Let $H'$ denote the $F$-limit of $\langle T_n \rangle_{n \in \mathbb{N}}$. By construction, $H'$ is a subgraph of $H$.

We first prove that the vertex sets coincide, that is, $V(H') = V(H)$. Let $v \in V(H)$. Then $v$ is incident with some edge $e_k$ for some $k \in \mathbb{Z}$. By the definition of the sequence, $e_k \in E(T_m)$ for all $m \geq |k|$, and hence $v \in V(T_m)$ for all $m \geq |k|$, since $T_{m}\subseteq T_{m+1}$. Thus, the set $\{n \in \mathbb{N} : v \in V(T_n)\}$ is cofinite, and therefore belongs to $F$. It follows that $v \in V(H')$, proving $V(H) \subseteq V(H')$. The reverse inclusion $V(H') \subseteq V(H)$ is immediate from the definition of $H'$. Hence, $V(H') = V(H)$.

We now prove that the edge sets coincide, that is, $E(H') = E(H)$. Let $e_k \in E(H)$, where $k \in \mathbb{Z}$. By construction, $e_k \in E(T_m)$ for all $m \geq |k|$. Hence, the set $\{n \in \mathbb{N} : e_k \in E(T_n)\}$ is cofinite and, since $F$ is a non-principal ultrafilter (and therefore contains every cofinite subset of $\mathbb{N}$), it follows that this set belongs to $F$. It follows that $e_k \in E(H')$, proving $E(H) \subseteq E(H')$. Again, the reverse inclusion $E(H') \subseteq E(H)$ follows from the definition of $H'$. Therefore, $E(H') = E(H)$.

Consequently, $H' = H$. Notice that $H$ is a topological cycle; in particular, it is not a topological tree.
	\end{example}

	\begin{proposition}[\cite{guilherme}]
\label{path1}
		Let $G$ be a locally finite graph and $H$ be a subgraph of $G$ which is the $F$-limit of finite paths between vertices $u$ and $v$. Then there is a topological path $P\subseteq H$ of $G$ between $u$ and $v$.
	\end{proposition}
	
Here we present a version of Proposition~\ref{path1} for $F$-limits of topological paths. Before that, it is useful to recall how the finite path case was established.
For the construction of topological paths, some elements of order theory will be used.
Let $G$ be locally finite and let $P \subseteq \|G\|$ be a topological path. Define a linear order on $P$ by setting $x \preceq y$ if $f^{-1}(x) \leq  f^{-1}(y)$, where $f : [0,1] \to P$ is a homeomorphism.
The proof will be omitted, as it is essentially identical to the argument for finite paths (see~\cite[Lemmas 6.4, 6.5 and 6.6]{guilherme}). We do not pursue these order-theoretic aspects further, as this would take us too far from the main focus of the paper.
	
		\begin{proposition}
			\label{path}
		Let $G$ be a locally finite graph and $H$ be a subgraph of $G$ which is the $F$-limit of topological paths between vertices $u$ and $v$. Then there is a topological path $P\subseteq H$ of $G$ between $u$ and $v$.
	\end{proposition}
	
\begin{lemma}
	\label{limitrees}
	Let $G$ be a locally finite connected graph and let $H\subseteq G$ be the $F$-limit of finite trees $T_n$, 
	where $S_n \subseteq V(T_n)$ and $S_n \subseteq S_{n+1}$, with each $S_n$ nonempty. 
	Then $H$ contains a topological tree $T$ such that 
	$$
	S = \bigcup_{n\in\mathbb{N}} S_n \subseteq V(T)
	$$
\end{lemma}
	\begin{proof}
Fix a non-principal ultrafilter $F$ over $\mathbb{N}$. We begin by showing that $H$ is path-connected. Let $u,v \in V(H)$. Then 
\[
A=\{n\in\mathbb{N}:u,v\in V(T_n)\}\in F
\]
by the definition of an $F$-limit. Since each $T_n$ is a tree, for every $n \in A$ there exists a path $P_n \subseteq T_n$ between $u$ and $v$. As $G$ is locally finite, Proposition~\ref{path1} implies that the $F$-limit of $\langle P_n\rangle_{n\in A}$ yields a topological path in $H$ between $u$ and $v$. Hence, $H$ is path-connected.

Consider an enumeration of the edges $E(H)=\{e_i\}_{i\in\mathbb{N}}$. The idea is to remove only edges that lie in topological circles of $G$ within $H$. If $e_0$ lies in a topological circle of $G$ in $H$, remove $e_0$ and set $\widetilde{T}_0 = H - e_0$; otherwise, let $\widetilde{T}_0 = H$. In general, for each $n \in \mathbb{N}$ with $n>0$, define
\[
\widetilde{T}_{n+1}=
\begin{cases}
	\widetilde{T}_n - e_n, & \text{if } e_n \text{ lies in a topological circle of } G \text{ in } \widetilde{T}_n,\\
	\widetilde{T}_n, & \text{otherwise}.
\end{cases}
\]
Note that $\widetilde{T}_{n+1} \subseteq \widetilde{T}_n$ for all $n \in \mathbb{N}$.

Define $T=\bigcap_{i\in\mathbb{N}} \widetilde{T}_i$. We claim that $T$ is a topological tree and that $S \subseteq V(T)$.

First, we show that each $\widetilde{T}_n$ is path-connected. Since $H$ is path-connected and $\widetilde{T}_0$ is obtained from $H$ by removing at most one edge lying in a topological circle, it follows that $\widetilde{T}_0$ remains path-connected. By the same argument, if $\widetilde{T}_n$ is path-connected, then so is $\widetilde{T}_{n+1}$. Hence, $\widetilde{T}_n$ is path-connected for every $n \in \mathbb{N}$.

We now show that $T$ is path-connected. Let $u,v \in V(T)$. For each $n \in \mathbb{N}$, since $\widetilde{T}_n$ is path-connected, there exists a topological path $P_n\subseteq\widetilde{T_n}$ between $u$ and $v$. By Proposition~\ref{path}, the $F$-limit of $\langle P_n\rangle_{n\in\mathbb{N}}$ contains a topological path $P$ between $u$ and $v$ in $T$. Thus, $T$ is path-connected.

It remains to show that $T$ contains no topological circles. Suppose, for a contradiction, that $T$ contains a topological circle. Then there exists an edge $e_k \in E(T)$ that lies in a topological circle of $T$, contradicting the construction of $\widetilde{T}_{k+1}$, since such an edge would have been removed at step $k$.

Finally, we verify that $S \subseteq V(T)$. Clearly, $S=\bigcup_{n\in\mathbb{N}} S_n \subseteq V(H)$, since for each $v \in S$ there exists $n \in \mathbb{N}$ such that $v \in S_n$, and as $S_n \subseteq S_{n+1}$, we have $v \in S_m$ for all $m \ge n$. Hence, $
\{m \in \mathbb{N} : v \in V(T_m)\} \in F,
$
and therefore $v \in V(H)$. Since we only remove edges in the construction, it follows that $v \in V(T)$, and thus $S \subseteq V(T)$.
	\end{proof}

	\begin{theorem}
	Conjectures~\ref{kriesell} and~\ref{conj} are equivalent. \label{localfin}
	\end{theorem}
	
	\begin{proof}
	Let $G$ be a connected locally finite graph and $S \subseteq V(G)$ be a $2k$-edge-connected subset. Consider an enumeration $S=\{v_n\}_{n\in\mathbb{N}}$ and, for each $n\in\mathbb{N}$, let $S_n=\{v_0,v_1,\dots, v_n\}$. Note that $S_n\subseteq S_{n+1}$ and each $S_n$ is nonempty.
	
	Let $G_n$ be the subgraph of $G$ obtained as the union of all edge-disjoint paths between every pair of vertices in $S_n$. Since $G$ is locally finite, it follows that $G_n$ is a finite connected subgraph. Moreover, $S_n$ is $2k$-edge-connected in $G_n$. Assuming Conjecture~\ref{kriesell}, for each $n\in\mathbb{N}$ there exist $k$ pairwise edge-disjoint $S_n$-Steiner trees
	$
	T^1_n, T^2_n, \dots, T^k_n
	$
	in $G_n$.
	
	Fix a non-principal ultrafilter $F$ over $\mathbb{N}$. For each $1\le i\le k$, let $H_i$ be the $F$-limit of $\langle T^i_n\rangle_{n\in\mathbb{N}}$. Note that $S\subseteq V(H_i)$, since $S_n\subseteq S_{n+1}$. By Lemma~\ref{limitrees}, we obtain a topological tree $T_i\subseteq H_i$ such that $S\subseteq V(T_i)$.
	
	We claim that the topological trees $T_1,T_2,\dots,T_k$ are pairwise edge-disjoint. Suppose, for a contradiction, that there exist distinct indices $i$ and $j$ and an edge $e\in E(T_i)\cap E(T_j)$. Then the set
$
	\{n\in\mathbb{N} : e\in E(T^i_n)\cap E(T^j_n)\}
	$ belongs to $F$, since $F$ is closed under finite intersections; in particular, it is nonempty. This contradicts the fact that $T^i_n$ and $T^j_n$ are edge-disjoint for every $n\in\mathbb{N}$. This proves that Conjecture \ref{conj} holds.

Conversely, assume that Conjecture \ref{conj} holds. Let $G$ be a finite connected graph and let $S \subseteq V(G)$ be a $2k$-edge-connected subset. Let $r=v_0v_1v_2\cdots$ be a ray such that $v_0=u \in V(G)$ and $v_n\notin V(G)$ for $n\neq 0$. Consider the graph $H=G\cup r$. Notice that $H$ is connected, infinite and locally finite.

Since $G\subseteq H$, the set $S$ is also $2k$-edge-connected in $H$. Therefore, there exist $k$ pairwise edge-disjoint topological $S$-Steiner trees $T_1,T_2,\ldots,T_k$ in $H$ with $\textsc{Etop}(H)$. 

We prove some claims concerning the ends and topological trees of $H$.
Recall that the degree of an end is the maximum number of pairwise vertex-disjoint rays representing it.

	\begin{claim}
		$H$ has exactly one end, and this end has degree $1$.
	\end{claim}
	
	\begin{proof}
		Notice that every ray of $H$ is either a tail of $r$ or starts in $G$, passes through $u$, and then continues along $r$, since $G$ is finite. Hence, the unique end of $H$ is $[r]\in \Omega(H)$. Moreover, this end has degree $1$, because every ray contains a tail of $r$.
	\end{proof}
	
	\begin{claim}
		\label{arv}
		A topological tree $T$ in $H$ with $\textsc{Etop}(H)$ is a tree of $H$.
	\end{claim}
	
	\begin{proof}
		Let $u,v\in V(T)$. Since $T$ is a topological tree, there exists a topological path $P\subseteq T$ between $u$ and $v$. Since $H$ has only one end of degree $1$, it follows that $P$ is a finite path. Therefore, $T$ is a connected subgraph. Hence, $T$ is a tree of $H$.
	\end{proof}
	By Claim~\ref{arv}, it follows that $T_1,T_2,\ldots,T_k$ are pairwise edge-disjoint $S$-Steiner trees of $H$. We show that
	$
	G\cap T_1,\; G\cap T_2,\; \ldots,\; G\cap T_k
	$
	are pairwise edge-disjoint $S$-Steiner trees of $G$. Since $G\cap T_i$ contains no finite cycle for every $1\leq i\leq k$, it suffices to show that $G\cap T_i$ is connected.
	
	Let $w,v\in V(G\cap T_i)$. Since $T_i$ is a tree, there exists a path $P\subseteq T_i$ between $w$ and $v$. Since $G$ is connected to $r$ only through the vertex $u$, it follows that $P\subseteq G$. Hence, $G\cap T_i$ is connected, and therefore it is a tree. Moreover, since each $T_i$ contains $S$ in its vertex set, we have
	$
	S\subseteq V(G\cap T_i)
	$
	for every $1\leq i\leq k$. Finally, because
	$
	E(T_i)\cap E(T_j)=\emptyset
	$
	whenever $i\neq j$, it follows that
	$
	E(G\cap T_i)\cap E(G\cap T_j)=\emptyset.
	$
	Therefore, $G$ contains $k$ pairwise edge-disjoint $S$-Steiner trees.
	Hence, Conjecture~\ref{kriesell} holds.
	\end{proof}

	\section{Countable case}
	\label{enum}

	For the proof of the countable case, we begin by defining the expanded graph $\widetilde{G}$ by preserving the vertices with finitely many incident edges and transforming each other vertex $v$ into a $2k$-ray (see Figure \ref{2kray}),  denoted by $r_v$, with the edge $e_i^v$ incident on the $i$-th element of it. We follow the construction of~\cite{guilherme} for multigraphs.
	
	\begin{definition}
Let $r \in \mathcal{R}(G)$ be a ray in a graph $G$, and let $k \in \mathbb{N}$ be a natural number. A $2k$-ray is obtained from $r$ by replacing each edge $e \in E(r)$ with a complete graph $K_{2k}$ and joining the endpoints of $e$ to all vertices of $K_{2k}$.

	\end{definition}
	
	\begin{figure}[ht]
		\centering
		\tikzset{every picture/.style={line width=0.75pt}} 
		
		\begin{tikzpicture}[x=0.75pt,y=0.75pt,yscale=-1,xscale=1,rotate=90]
			
			\draw    (100,121) -- (100,139.6) ;
			\draw [shift={(100,139.6)}, rotate = 90] [color={rgb, 255:red, 0; green, 0; blue, 0 }  ][fill={rgb, 255:red, 0; green, 0; blue, 0 }  ][line width=0.75]      (0, 0) circle [x radius= 3.35, y radius= 3.35]   ;
			\draw [shift={(100,121)}, rotate = 90] [color={rgb, 255:red, 0; green, 0; blue, 0 }  ][fill={rgb, 255:red, 0; green, 0; blue, 0 }  ][line width=0.75]      (0, 0) circle [x radius= 3.35, y radius= 3.35]   ;
			\draw    (120,121) -- (120,139.6) ;
			\draw [shift={(120,139.6)}, rotate = 90] [color={rgb, 255:red, 0; green, 0; blue, 0 }  ][fill={rgb, 255:red, 0; green, 0; blue, 0 }  ][line width=0.75]      (0, 0) circle [x radius= 3.35, y radius= 3.35]   ;
			\draw [shift={(120,121)}, rotate = 90] [color={rgb, 255:red, 0; green, 0; blue, 0 }  ][fill={rgb, 255:red, 0; green, 0; blue, 0 }  ][line width=0.75]      (0, 0) circle [x radius= 3.35, y radius= 3.35]   ;
			\draw    (100,121) -- (120,121) ;
			\draw [shift={(120,121)}, rotate = 0] [color={rgb, 255:red, 0; green, 0; blue, 0 }  ][fill={rgb, 255:red, 0; green, 0; blue, 0 }  ][line width=0.75]      (0, 0) circle [x radius= 3.35, y radius= 3.35]   ;
			\draw [shift={(100,121)}, rotate = 0] [color={rgb, 255:red, 0; green, 0; blue, 0 }  ][fill={rgb, 255:red, 0; green, 0; blue, 0 }  ][line width=0.75]      (0, 0) circle [x radius= 3.35, y radius= 3.35]   ;
			\draw    (100,139.6) -- (120,139.6) ;
			\draw [shift={(120,139.6)}, rotate = 0] [color={rgb, 255:red, 0; green, 0; blue, 0 }  ][fill={rgb, 255:red, 0; green, 0; blue, 0 }  ][line width=0.75]      (0, 0) circle [x radius= 3.35, y radius= 3.35]   ;
			\draw [shift={(100,139.6)}, rotate = 0] [color={rgb, 255:red, 0; green, 0; blue, 0 }  ][fill={rgb, 255:red, 0; green, 0; blue, 0 }  ][line width=0.75]      (0, 0) circle [x radius= 3.35, y radius= 3.35]   ;
			\draw    (100,121) -- (120,139.6) ;
			\draw [shift={(120,139.6)}, rotate = 42.92] [color={rgb, 255:red, 0; green, 0; blue, 0 }  ][fill={rgb, 255:red, 0; green, 0; blue, 0 }  ][line width=0.75]      (0, 0) circle [x radius= 3.35, y radius= 3.35]   ;
			\draw [shift={(100,121)}, rotate = 42.92] [color={rgb, 255:red, 0; green, 0; blue, 0 }  ][fill={rgb, 255:red, 0; green, 0; blue, 0 }  ][line width=0.75]      (0, 0) circle [x radius= 3.35, y radius= 3.35]   ;
			\draw    (120,121) -- (100,139.6) ;
			\draw [shift={(100,139.6)}, rotate = 137.08] [color={rgb, 255:red, 0; green, 0; blue, 0 }  ][fill={rgb, 255:red, 0; green, 0; blue, 0 }  ][line width=0.75]      (0, 0) circle [x radius= 3.35, y radius= 3.35]   ;
			\draw [shift={(120,121)}, rotate = 137.08] [color={rgb, 255:red, 0; green, 0; blue, 0 }  ][fill={rgb, 255:red, 0; green, 0; blue, 0 }  ][line width=0.75]      (0, 0) circle [x radius= 3.35, y radius= 3.35]   ;
			\draw    (100,74) -- (100,92.6) ;
			\draw [shift={(100,92.6)}, rotate = 90] [color={rgb, 255:red, 0; green, 0; blue, 0 }  ][fill={rgb, 255:red, 0; green, 0; blue, 0 }  ][line width=0.75]      (0, 0) circle [x radius= 3.35, y radius= 3.35]   ;
			\draw [shift={(100,74)}, rotate = 90] [color={rgb, 255:red, 0; green, 0; blue, 0 }  ][fill={rgb, 255:red, 0; green, 0; blue, 0 }  ][line width=0.75]      (0, 0) circle [x radius= 3.35, y radius= 3.35]   ;
			\draw    (120,74) -- (120,92.6) ;
			\draw [shift={(120,92.6)}, rotate = 90] [color={rgb, 255:red, 0; green, 0; blue, 0 }  ][fill={rgb, 255:red, 0; green, 0; blue, 0 }  ][line width=0.75]      (0, 0) circle [x radius= 3.35, y radius= 3.35]   ;
			\draw [shift={(120,74)}, rotate = 90] [color={rgb, 255:red, 0; green, 0; blue, 0 }  ][fill={rgb, 255:red, 0; green, 0; blue, 0 }  ][line width=0.75]      (0, 0) circle [x radius= 3.35, y radius= 3.35]   ;
			\draw    (100,74) -- (120,74) ;
			\draw [shift={(120,74)}, rotate = 0] [color={rgb, 255:red, 0; green, 0; blue, 0 }  ][fill={rgb, 255:red, 0; green, 0; blue, 0 }  ][line width=0.75]      (0, 0) circle [x radius= 3.35, y radius= 3.35]   ;
			\draw [shift={(100,74)}, rotate = 0] [color={rgb, 255:red, 0; green, 0; blue, 0 }  ][fill={rgb, 255:red, 0; green, 0; blue, 0 }  ][line width=0.75]      (0, 0) circle [x radius= 3.35, y radius= 3.35]   ;
			\draw    (100,92.6) -- (120,92.6) ;
			\draw [shift={(120,92.6)}, rotate = 0] [color={rgb, 255:red, 0; green, 0; blue, 0 }  ][fill={rgb, 255:red, 0; green, 0; blue, 0 }  ][line width=0.75]      (0, 0) circle [x radius= 3.35, y radius= 3.35]   ;
			\draw [shift={(100,92.6)}, rotate = 0] [color={rgb, 255:red, 0; green, 0; blue, 0 }  ][fill={rgb, 255:red, 0; green, 0; blue, 0 }  ][line width=0.75]      (0, 0) circle [x radius= 3.35, y radius= 3.35]   ;
			\draw    (100,74) -- (120,92.6) ;
			\draw [shift={(120,92.6)}, rotate = 42.92] [color={rgb, 255:red, 0; green, 0; blue, 0 }  ][fill={rgb, 255:red, 0; green, 0; blue, 0 }  ][line width=0.75]      (0, 0) circle [x radius= 3.35, y radius= 3.35]   ;
			\draw [shift={(100,74)}, rotate = 42.92] [color={rgb, 255:red, 0; green, 0; blue, 0 }  ][fill={rgb, 255:red, 0; green, 0; blue, 0 }  ][line width=0.75]      (0, 0) circle [x radius= 3.35, y radius= 3.35]   ;
			\draw    (120,74) -- (100,92.6) ;
			\draw [shift={(100,92.6)}, rotate = 137.08] [color={rgb, 255:red, 0; green, 0; blue, 0 }  ][fill={rgb, 255:red, 0; green, 0; blue, 0 }  ][line width=0.75]      (0, 0) circle [x radius= 3.35, y radius= 3.35]   ;
			\draw [shift={(120,74)}, rotate = 137.08] [color={rgb, 255:red, 0; green, 0; blue, 0 }  ][fill={rgb, 255:red, 0; green, 0; blue, 0 }  ][line width=0.75]      (0, 0) circle [x radius= 3.35, y radius= 3.35]   ;
			\draw    (100,170) -- (100,188.6) ;
			\draw [shift={(100,188.6)}, rotate = 90] [color={rgb, 255:red, 0; green, 0; blue, 0 }  ][fill={rgb, 255:red, 0; green, 0; blue, 0 }  ][line width=0.75]      (0, 0) circle [x radius= 3.35, y radius= 3.35]   ;
			\draw [shift={(100,170)}, rotate = 90] [color={rgb, 255:red, 0; green, 0; blue, 0 }  ][fill={rgb, 255:red, 0; green, 0; blue, 0 }  ][line width=0.75]      (0, 0) circle [x radius= 3.35, y radius= 3.35]   ;
			\draw    (120,170) -- (120,188.6) ;
			\draw [shift={(120,188.6)}, rotate = 90] [color={rgb, 255:red, 0; green, 0; blue, 0 }  ][fill={rgb, 255:red, 0; green, 0; blue, 0 }  ][line width=0.75]      (0, 0) circle [x radius= 3.35, y radius= 3.35]   ;
			\draw [shift={(120,170)}, rotate = 90] [color={rgb, 255:red, 0; green, 0; blue, 0 }  ][fill={rgb, 255:red, 0; green, 0; blue, 0 }  ][line width=0.75]      (0, 0) circle [x radius= 3.35, y radius= 3.35]   ;
			\draw    (100,170) -- (120,170) ;
			\draw [shift={(120,170)}, rotate = 0] [color={rgb, 255:red, 0; green, 0; blue, 0 }  ][fill={rgb, 255:red, 0; green, 0; blue, 0 }  ][line width=0.75]      (0, 0) circle [x radius= 3.35, y radius= 3.35]   ;
			\draw [shift={(100,170)}, rotate = 0] [color={rgb, 255:red, 0; green, 0; blue, 0 }  ][fill={rgb, 255:red, 0; green, 0; blue, 0 }  ][line width=0.75]      (0, 0) circle [x radius= 3.35, y radius= 3.35]   ;
			\draw    (100,188.6) -- (120,188.6) ;
			\draw [shift={(120,188.6)}, rotate = 0] [color={rgb, 255:red, 0; green, 0; blue, 0 }  ][fill={rgb, 255:red, 0; green, 0; blue, 0 }  ][line width=0.75]      (0, 0) circle [x radius= 3.35, y radius= 3.35]   ;
			\draw [shift={(100,188.6)}, rotate = 0] [color={rgb, 255:red, 0; green, 0; blue, 0 }  ][fill={rgb, 255:red, 0; green, 0; blue, 0 }  ][line width=0.75]      (0, 0) circle [x radius= 3.35, y radius= 3.35]   ;
			\draw    (100,170) -- (120,188.6) ;
			\draw [shift={(120,188.6)}, rotate = 42.92] [color={rgb, 255:red, 0; green, 0; blue, 0 }  ][fill={rgb, 255:red, 0; green, 0; blue, 0 }  ][line width=0.75]      (0, 0) circle [x radius= 3.35, y radius= 3.35]   ;
			\draw [shift={(100,170)}, rotate = 42.92] [color={rgb, 255:red, 0; green, 0; blue, 0 }  ][fill={rgb, 255:red, 0; green, 0; blue, 0 }  ][line width=0.75]      (0, 0) circle [x radius= 3.35, y radius= 3.35]   ;
			\draw    (120,170) -- (100,188.6) ;
			\draw [shift={(100,188.6)}, rotate = 137.08] [color={rgb, 255:red, 0; green, 0; blue, 0 }  ][fill={rgb, 255:red, 0; green, 0; blue, 0 }  ][line width=0.75]      (0, 0) circle [x radius= 3.35, y radius= 3.35]   ;
			\draw [shift={(120,170)}, rotate = 137.08] [color={rgb, 255:red, 0; green, 0; blue, 0 }  ][fill={rgb, 255:red, 0; green, 0; blue, 0 }  ][line width=0.75]      (0, 0) circle [x radius= 3.35, y radius= 3.35]   ;
			\draw    (145,106) ;
			\draw [shift={(145,106)}, rotate = 0] [color={rgb, 255:red, 0; green, 0; blue, 0 }  ][fill={rgb, 255:red, 0; green, 0; blue, 0 }  ][line width=0.75]      (0, 0) circle [x radius= 3.35, y radius= 3.35]   ;
			\draw [shift={(145,106)}, rotate = 0] [color={rgb, 255:red, 0; green, 0; blue, 0 }  ][fill={rgb, 255:red, 0; green, 0; blue, 0 }  ][line width=0.75]      (0, 0) circle [x radius= 3.35, y radius= 3.35]   ;
			\draw    (145,54.6) ;
			\draw [shift={(145,54.6)}, rotate = 0] [color={rgb, 255:red, 0; green, 0; blue, 0 }  ][fill={rgb, 255:red, 0; green, 0; blue, 0 }  ][line width=0.75]      (0, 0) circle [x radius= 3.35, y radius= 3.35]   ;
			\draw [shift={(145,54.6)}, rotate = 0] [color={rgb, 255:red, 0; green, 0; blue, 0 }  ][fill={rgb, 255:red, 0; green, 0; blue, 0 }  ][line width=0.75]      (0, 0) circle [x radius= 3.35, y radius= 3.35]   ;
			\draw    (145,155.6) ;
			\draw [shift={(145,155.6)}, rotate = 0] [color={rgb, 255:red, 0; green, 0; blue, 0 }  ][fill={rgb, 255:red, 0; green, 0; blue, 0 }  ][line width=0.75]      (0, 0) circle [x radius= 3.35, y radius= 3.35]   ;
			\draw [shift={(145,155.6)}, rotate = 0] [color={rgb, 255:red, 0; green, 0; blue, 0 }  ][fill={rgb, 255:red, 0; green, 0; blue, 0 }  ][line width=0.75]      (0, 0) circle [x radius= 3.35, y radius= 3.35]   ;
			\draw    (145,205.6) ;
			\draw [shift={(145,205.6)}, rotate = 0] [color={rgb, 255:red, 0; green, 0; blue, 0 }  ][fill={rgb, 255:red, 0; green, 0; blue, 0 }  ][line width=0.75]      (0, 0) circle [x radius= 3.35, y radius= 3.35]   ;
			\draw [shift={(145,205.6)}, rotate = 0] [color={rgb, 255:red, 0; green, 0; blue, 0 }  ][fill={rgb, 255:red, 0; green, 0; blue, 0 }  ][line width=0.75]      (0, 0) circle [x radius= 3.35, y radius= 3.35]   ;
			\draw    (100,74) -- (145,54.6) ;
			\draw    (120,74) -- (145,54.6) ;
			\draw    (120,92.6) -- (145,54.6) ;
			\draw    (100,92.6) -- (145,54.6) ;
			\draw    (120,74) -- (145,106) ;
			\draw    (120,92.6) -- (145,106) ;
			\draw    (100,92.6) -- (145,106) ;
			\draw    (100,74) -- (145,106) ;
			\draw    (100,121) -- (145,106) ;
			\draw    (120,121) -- (145,106) ;
			\draw    (120,139.6) -- (145,106) ;
			\draw    (100,139.6) -- (145,106) ;
			\draw    (120,139.6) -- (145,155.6) ;
			\draw    (100,139.6) -- (145,155.6) ;
			\draw    (120,121) -- (145,155.6) ;
			\draw    (100,121) -- (145,155.6) ;
			\draw    (100,170) -- (145,155.6) ;
			\draw    (120,170) -- (145,155.6) ;
			\draw    (120,188.6) -- (145,155.6) ;
			\draw    (100,188.6) -- (145,155.6) ;
			\draw    (120,170) -- (145,205.6) ;
			\draw    (120,188.6) -- (145,205.6) ;
			\draw    (100,188.6) -- (145,205.6) ;
			\draw    (100,170) -- (145,205.6) ;
			\draw  [dash pattern={on 0.84pt off 2.51pt}]  (145,45) -- (145,33.6) ;

		\end{tikzpicture}
		
		\caption{A $4$-ray}
		\end{figure}
	
Let $G$ be a countable graph. For each vertex $v \in V(G)$, consider  $E(v)$ to be the set of edges incident on $v$. If $v$ has infinitely many incident edges, enumerate them by
$
E(v) = \{ e_i^v : i \in \mathbb{N} \}
$. Let $A$ be the set of vertices of $G$ of finite degree, and let  
$ B= V(G) \setminus A$ be the set of vertices of infinite degree.
The vertex set of  the expanded graph $\widetilde{G}$ of $G$ is defined as
$$
V(\widetilde{G}) = A \cup \bigl( B\times \mathbb{N}\bigr)\cup (\bigcup_{v\in B} \{w_n^v:n\in\mathbb{N}\}) $$

Let $u,v \in V(\widetilde{G})$ and $k\in\mathbb{N}$ be a natural number. We define the edges $E(\widetilde{G})$ of $\widetilde{G}$ as follows:

\begin{itemize}
	\item If $u,v \in A$, then $uv \in E(\widetilde{G})$ whenever $uv \in E(G)$.
	
	\item If $u\in A$ and $v \in B$, then $u(v,n) \in E(\widetilde{G})$ if $uv=e_n^v$.

	\item If $u,v \in B$ and $uv=e_n^u=e_m^v$, then 
$(u,n)(v,m) \in E(\widetilde{G})$

\item If $i,j\in\mathbb{N}$ and $2kn\leq i,j<2k(n+1)$ for a natural number $n\in\mathbb{N}$, then $w_i^vw_j^v\in E(\widetilde{G})$ and $(v,n)w_i^v, (v,n+1)w_i^v\in E(\widetilde{G})$ for $v\in B$.

\end{itemize}

For a vertex $v \in B$ of infinite degree, the induced subgraph
$G\bigl[\{(v,n),\, w_n^v : n \in \mathbb{N}\}\bigr]$ is a $2k$-ray.
Figure~\ref{2kray} illustrates the expanded graph of an infinite star
for $k = 2$.
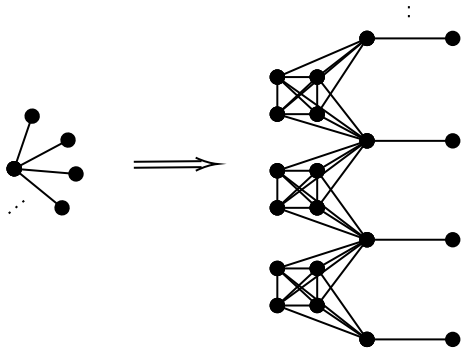
\begin{figure}[htbp]
	\centering
	\tikzset{every picture/.style={line width=0.75pt}} 
	
	\begin{tikzpicture}[x=0.75pt,y=0.75pt,yscale=-1,xscale=1]
		
		\draw    (232,123) -- (232,141.6) ;
		\draw [shift={(232,141.6)}, rotate = 90] [color={rgb, 255:red, 0; green, 0; blue, 0 }  ][fill={rgb, 255:red, 0; green, 0; blue, 0 }  ][line width=0.75]      (0, 0) circle [x radius= 3.35, y radius= 3.35]   ;
		\draw [shift={(232,123)}, rotate = 90] [color={rgb, 255:red, 0; green, 0; blue, 0 }  ][fill={rgb, 255:red, 0; green, 0; blue, 0 }  ][line width=0.75]      (0, 0) circle [x radius= 3.35, y radius= 3.35]   ;
		\draw    (252,123) -- (252,141.6) ;
		\draw [shift={(252,141.6)}, rotate = 90] [color={rgb, 255:red, 0; green, 0; blue, 0 }  ][fill={rgb, 255:red, 0; green, 0; blue, 0 }  ][line width=0.75]      (0, 0) circle [x radius= 3.35, y radius= 3.35]   ;
		\draw [shift={(252,123)}, rotate = 90] [color={rgb, 255:red, 0; green, 0; blue, 0 }  ][fill={rgb, 255:red, 0; green, 0; blue, 0 }  ][line width=0.75]      (0, 0) circle [x radius= 3.35, y radius= 3.35]   ;
		\draw    (232,123) -- (252,123) ;
		\draw [shift={(252,123)}, rotate = 0] [color={rgb, 255:red, 0; green, 0; blue, 0 }  ][fill={rgb, 255:red, 0; green, 0; blue, 0 }  ][line width=0.75]      (0, 0) circle [x radius= 3.35, y radius= 3.35]   ;
		\draw [shift={(232,123)}, rotate = 0] [color={rgb, 255:red, 0; green, 0; blue, 0 }  ][fill={rgb, 255:red, 0; green, 0; blue, 0 }  ][line width=0.75]      (0, 0) circle [x radius= 3.35, y radius= 3.35]   ;
		\draw    (232,141.6) -- (252,141.6) ;
		\draw [shift={(252,141.6)}, rotate = 0] [color={rgb, 255:red, 0; green, 0; blue, 0 }  ][fill={rgb, 255:red, 0; green, 0; blue, 0 }  ][line width=0.75]      (0, 0) circle [x radius= 3.35, y radius= 3.35]   ;
		\draw [shift={(232,141.6)}, rotate = 0] [color={rgb, 255:red, 0; green, 0; blue, 0 }  ][fill={rgb, 255:red, 0; green, 0; blue, 0 }  ][line width=0.75]      (0, 0) circle [x radius= 3.35, y radius= 3.35]   ;
		\draw    (232,123) -- (252,141.6) ;
		\draw [shift={(252,141.6)}, rotate = 42.92] [color={rgb, 255:red, 0; green, 0; blue, 0 }  ][fill={rgb, 255:red, 0; green, 0; blue, 0 }  ][line width=0.75]      (0, 0) circle [x radius= 3.35, y radius= 3.35]   ;
		\draw [shift={(232,123)}, rotate = 42.92] [color={rgb, 255:red, 0; green, 0; blue, 0 }  ][fill={rgb, 255:red, 0; green, 0; blue, 0 }  ][line width=0.75]      (0, 0) circle [x radius= 3.35, y radius= 3.35]   ;
		\draw    (252,123) -- (232,141.6) ;
		\draw [shift={(232,141.6)}, rotate = 137.08] [color={rgb, 255:red, 0; green, 0; blue, 0 }  ][fill={rgb, 255:red, 0; green, 0; blue, 0 }  ][line width=0.75]      (0, 0) circle [x radius= 3.35, y radius= 3.35]   ;
		\draw [shift={(252,123)}, rotate = 137.08] [color={rgb, 255:red, 0; green, 0; blue, 0 }  ][fill={rgb, 255:red, 0; green, 0; blue, 0 }  ][line width=0.75]      (0, 0) circle [x radius= 3.35, y radius= 3.35]   ;
		\draw    (232,76) -- (232,94.6) ;
		\draw [shift={(232,94.6)}, rotate = 90] [color={rgb, 255:red, 0; green, 0; blue, 0 }  ][fill={rgb, 255:red, 0; green, 0; blue, 0 }  ][line width=0.75]      (0, 0) circle [x radius= 3.35, y radius= 3.35]   ;
		\draw [shift={(232,76)}, rotate = 90] [color={rgb, 255:red, 0; green, 0; blue, 0 }  ][fill={rgb, 255:red, 0; green, 0; blue, 0 }  ][line width=0.75]      (0, 0) circle [x radius= 3.35, y radius= 3.35]   ;
		\draw    (252,76) -- (252,94.6) ;
		\draw [shift={(252,94.6)}, rotate = 90] [color={rgb, 255:red, 0; green, 0; blue, 0 }  ][fill={rgb, 255:red, 0; green, 0; blue, 0 }  ][line width=0.75]      (0, 0) circle [x radius= 3.35, y radius= 3.35]   ;
		\draw [shift={(252,76)}, rotate = 90] [color={rgb, 255:red, 0; green, 0; blue, 0 }  ][fill={rgb, 255:red, 0; green, 0; blue, 0 }  ][line width=0.75]      (0, 0) circle [x radius= 3.35, y radius= 3.35]   ;
		\draw    (232,76) -- (252,76) ;
		\draw [shift={(252,76)}, rotate = 0] [color={rgb, 255:red, 0; green, 0; blue, 0 }  ][fill={rgb, 255:red, 0; green, 0; blue, 0 }  ][line width=0.75]      (0, 0) circle [x radius= 3.35, y radius= 3.35]   ;
		\draw [shift={(232,76)}, rotate = 0] [color={rgb, 255:red, 0; green, 0; blue, 0 }  ][fill={rgb, 255:red, 0; green, 0; blue, 0 }  ][line width=0.75]      (0, 0) circle [x radius= 3.35, y radius= 3.35]   ;
		\draw    (232,94.6) -- (252,94.6) ;
		\draw [shift={(252,94.6)}, rotate = 0] [color={rgb, 255:red, 0; green, 0; blue, 0 }  ][fill={rgb, 255:red, 0; green, 0; blue, 0 }  ][line width=0.75]      (0, 0) circle [x radius= 3.35, y radius= 3.35]   ;
		\draw [shift={(232,94.6)}, rotate = 0] [color={rgb, 255:red, 0; green, 0; blue, 0 }  ][fill={rgb, 255:red, 0; green, 0; blue, 0 }  ][line width=0.75]      (0, 0) circle [x radius= 3.35, y radius= 3.35]   ;
		\draw    (232,76) -- (252,94.6) ;
		\draw [shift={(252,94.6)}, rotate = 42.92] [color={rgb, 255:red, 0; green, 0; blue, 0 }  ][fill={rgb, 255:red, 0; green, 0; blue, 0 }  ][line width=0.75]      (0, 0) circle [x radius= 3.35, y radius= 3.35]   ;
		\draw [shift={(232,76)}, rotate = 42.92] [color={rgb, 255:red, 0; green, 0; blue, 0 }  ][fill={rgb, 255:red, 0; green, 0; blue, 0 }  ][line width=0.75]      (0, 0) circle [x radius= 3.35, y radius= 3.35]   ;
		\draw    (252,76) -- (232,94.6) ;
		\draw [shift={(232,94.6)}, rotate = 137.08] [color={rgb, 255:red, 0; green, 0; blue, 0 }  ][fill={rgb, 255:red, 0; green, 0; blue, 0 }  ][line width=0.75]      (0, 0) circle [x radius= 3.35, y radius= 3.35]   ;
		\draw [shift={(252,76)}, rotate = 137.08] [color={rgb, 255:red, 0; green, 0; blue, 0 }  ][fill={rgb, 255:red, 0; green, 0; blue, 0 }  ][line width=0.75]      (0, 0) circle [x radius= 3.35, y radius= 3.35]   ;
		\draw    (232,172) -- (232,190.6) ;
		\draw [shift={(232,190.6)}, rotate = 90] [color={rgb, 255:red, 0; green, 0; blue, 0 }  ][fill={rgb, 255:red, 0; green, 0; blue, 0 }  ][line width=0.75]      (0, 0) circle [x radius= 3.35, y radius= 3.35]   ;
		\draw [shift={(232,172)}, rotate = 90] [color={rgb, 255:red, 0; green, 0; blue, 0 }  ][fill={rgb, 255:red, 0; green, 0; blue, 0 }  ][line width=0.75]      (0, 0) circle [x radius= 3.35, y radius= 3.35]   ;
		\draw    (252,172) -- (252,190.6) ;
		\draw [shift={(252,190.6)}, rotate = 90] [color={rgb, 255:red, 0; green, 0; blue, 0 }  ][fill={rgb, 255:red, 0; green, 0; blue, 0 }  ][line width=0.75]      (0, 0) circle [x radius= 3.35, y radius= 3.35]   ;
		\draw [shift={(252,172)}, rotate = 90] [color={rgb, 255:red, 0; green, 0; blue, 0 }  ][fill={rgb, 255:red, 0; green, 0; blue, 0 }  ][line width=0.75]      (0, 0) circle [x radius= 3.35, y radius= 3.35]   ;
		\draw    (232,172) -- (252,172) ;
		\draw [shift={(252,172)}, rotate = 0] [color={rgb, 255:red, 0; green, 0; blue, 0 }  ][fill={rgb, 255:red, 0; green, 0; blue, 0 }  ][line width=0.75]      (0, 0) circle [x radius= 3.35, y radius= 3.35]   ;
		\draw [shift={(232,172)}, rotate = 0] [color={rgb, 255:red, 0; green, 0; blue, 0 }  ][fill={rgb, 255:red, 0; green, 0; blue, 0 }  ][line width=0.75]      (0, 0) circle [x radius= 3.35, y radius= 3.35]   ;
		\draw    (232,190.6) -- (252,190.6) ;
		\draw [shift={(252,190.6)}, rotate = 0] [color={rgb, 255:red, 0; green, 0; blue, 0 }  ][fill={rgb, 255:red, 0; green, 0; blue, 0 }  ][line width=0.75]      (0, 0) circle [x radius= 3.35, y radius= 3.35]   ;
		\draw [shift={(232,190.6)}, rotate = 0] [color={rgb, 255:red, 0; green, 0; blue, 0 }  ][fill={rgb, 255:red, 0; green, 0; blue, 0 }  ][line width=0.75]      (0, 0) circle [x radius= 3.35, y radius= 3.35]   ;
		\draw    (232,172) -- (252,190.6) ;
		\draw [shift={(252,190.6)}, rotate = 42.92] [color={rgb, 255:red, 0; green, 0; blue, 0 }  ][fill={rgb, 255:red, 0; green, 0; blue, 0 }  ][line width=0.75]      (0, 0) circle [x radius= 3.35, y radius= 3.35]   ;
		\draw [shift={(232,172)}, rotate = 42.92] [color={rgb, 255:red, 0; green, 0; blue, 0 }  ][fill={rgb, 255:red, 0; green, 0; blue, 0 }  ][line width=0.75]      (0, 0) circle [x radius= 3.35, y radius= 3.35]   ;
		\draw    (252,172) -- (232,190.6) ;
		\draw [shift={(232,190.6)}, rotate = 137.08] [color={rgb, 255:red, 0; green, 0; blue, 0 }  ][fill={rgb, 255:red, 0; green, 0; blue, 0 }  ][line width=0.75]      (0, 0) circle [x radius= 3.35, y radius= 3.35]   ;
		\draw [shift={(252,172)}, rotate = 137.08] [color={rgb, 255:red, 0; green, 0; blue, 0 }  ][fill={rgb, 255:red, 0; green, 0; blue, 0 }  ][line width=0.75]      (0, 0) circle [x radius= 3.35, y radius= 3.35]   ;
		\draw    (277,108) ;
		\draw [shift={(277,108)}, rotate = 0] [color={rgb, 255:red, 0; green, 0; blue, 0 }  ][fill={rgb, 255:red, 0; green, 0; blue, 0 }  ][line width=0.75]      (0, 0) circle [x radius= 3.35, y radius= 3.35]   ;
		\draw [shift={(277,108)}, rotate = 0] [color={rgb, 255:red, 0; green, 0; blue, 0 }  ][fill={rgb, 255:red, 0; green, 0; blue, 0 }  ][line width=0.75]      (0, 0) circle [x radius= 3.35, y radius= 3.35]   ;
		\draw    (277,56.6) ;
		\draw [shift={(277,56.6)}, rotate = 0] [color={rgb, 255:red, 0; green, 0; blue, 0 }  ][fill={rgb, 255:red, 0; green, 0; blue, 0 }  ][line width=0.75]      (0, 0) circle [x radius= 3.35, y radius= 3.35]   ;
		\draw [shift={(277,56.6)}, rotate = 0] [color={rgb, 255:red, 0; green, 0; blue, 0 }  ][fill={rgb, 255:red, 0; green, 0; blue, 0 }  ][line width=0.75]      (0, 0) circle [x radius= 3.35, y radius= 3.35]   ;
		\draw    (277,157.6) ;
		\draw [shift={(277,157.6)}, rotate = 0] [color={rgb, 255:red, 0; green, 0; blue, 0 }  ][fill={rgb, 255:red, 0; green, 0; blue, 0 }  ][line width=0.75]      (0, 0) circle [x radius= 3.35, y radius= 3.35]   ;
		\draw [shift={(277,157.6)}, rotate = 0] [color={rgb, 255:red, 0; green, 0; blue, 0 }  ][fill={rgb, 255:red, 0; green, 0; blue, 0 }  ][line width=0.75]      (0, 0) circle [x radius= 3.35, y radius= 3.35]   ;
		\draw    (277,207.6) ;
		\draw [shift={(277,207.6)}, rotate = 0] [color={rgb, 255:red, 0; green, 0; blue, 0 }  ][fill={rgb, 255:red, 0; green, 0; blue, 0 }  ][line width=0.75]      (0, 0) circle [x radius= 3.35, y radius= 3.35]   ;
		\draw [shift={(277,207.6)}, rotate = 0] [color={rgb, 255:red, 0; green, 0; blue, 0 }  ][fill={rgb, 255:red, 0; green, 0; blue, 0 }  ][line width=0.75]      (0, 0) circle [x radius= 3.35, y radius= 3.35]   ;
		\draw    (232,76) -- (277,56.6) ;
		\draw    (252,76) -- (277,56.6) ;
		\draw    (252,94.6) -- (277,56.6) ;
		\draw    (232,94.6) -- (277,56.6) ;
		\draw    (252,76) -- (277,108) ;
		\draw    (252,94.6) -- (277,108) ;
		\draw    (232,94.6) -- (277,108) ;
		\draw    (232,76) -- (277,108) ;
		\draw    (232,123) -- (277,108) ;
		\draw    (252,123) -- (277,108) ;
		\draw    (252,141.6) -- (277,108) ;
		\draw    (232,141.6) -- (277,108) ;
		\draw    (252,141.6) -- (277,157.6) ;
		\draw    (232,141.6) -- (277,157.6) ;
		\draw    (252,123) -- (277,157.6) ;
		\draw    (232,123) -- (277,157.6) ;
		\draw    (232,172) -- (277,157.6) ;
		\draw    (252,172) -- (277,157.6) ;
		\draw    (252,190.6) -- (277,157.6) ;
		\draw    (232,190.6) -- (277,157.6) ;
		\draw    (252,172) -- (277,207.6) ;
		\draw    (252,190.6) -- (277,207.6) ;
		\draw    (232,190.6) -- (277,207.6) ;
		\draw    (232,172) -- (277,207.6) ;
		\draw  [dash pattern={on 0.84pt off 2.51pt}]  (298,46) -- (298,39.6) -- (298,34.6) ;
		\draw    (100,122) -- (127,107.6) ;
		\draw [shift={(127,107.6)}, rotate = 331.93] [color={rgb, 255:red, 0; green, 0; blue, 0 }  ][fill={rgb, 255:red, 0; green, 0; blue, 0 }  ][line width=0.75]      (0, 0) circle [x radius= 3.35, y radius= 3.35]   ;
		\draw [shift={(100,122)}, rotate = 331.93] [color={rgb, 255:red, 0; green, 0; blue, 0 }  ][fill={rgb, 255:red, 0; green, 0; blue, 0 }  ][line width=0.75]      (0, 0) circle [x radius= 3.35, y radius= 3.35]   ;
		\draw    (100,122) -- (131,124.6) ;
		\draw [shift={(131,124.6)}, rotate = 4.79] [color={rgb, 255:red, 0; green, 0; blue, 0 }  ][fill={rgb, 255:red, 0; green, 0; blue, 0 }  ][line width=0.75]      (0, 0) circle [x radius= 3.35, y radius= 3.35]   ;
		\draw [shift={(100,122)}, rotate = 4.79] [color={rgb, 255:red, 0; green, 0; blue, 0 }  ][fill={rgb, 255:red, 0; green, 0; blue, 0 }  ][line width=0.75]      (0, 0) circle [x radius= 3.35, y radius= 3.35]   ;
		\draw    (100,122) -- (124,141.6) ;
		\draw [shift={(124,141.6)}, rotate = 39.24] [color={rgb, 255:red, 0; green, 0; blue, 0 }  ][fill={rgb, 255:red, 0; green, 0; blue, 0 }  ][line width=0.75]      (0, 0) circle [x radius= 3.35, y radius= 3.35]   ;
		\draw [shift={(100,122)}, rotate = 39.24] [color={rgb, 255:red, 0; green, 0; blue, 0 }  ][fill={rgb, 255:red, 0; green, 0; blue, 0 }  ][line width=0.75]      (0, 0) circle [x radius= 3.35, y radius= 3.35]   ;
		\draw    (100,122) -- (109,95.6) ;
		\draw [shift={(109,95.6)}, rotate = 288.82] [color={rgb, 255:red, 0; green, 0; blue, 0 }  ][fill={rgb, 255:red, 0; green, 0; blue, 0 }  ][line width=0.75]      (0, 0) circle [x radius= 3.35, y radius= 3.35]   ;
		\draw [shift={(100,122)}, rotate = 288.82] [color={rgb, 255:red, 0; green, 0; blue, 0 }  ][fill={rgb, 255:red, 0; green, 0; blue, 0 }  ][line width=0.75]      (0, 0) circle [x radius= 3.35, y radius= 3.35]   ;
		\draw  [dash pattern={on 0.84pt off 2.51pt}]  (105,138) -- (97,144.6) ;
		\draw    (159.99,118.5) -- (193.99,118.18)(160.01,121.5) -- (194.01,121.18) ;
		\draw [shift={(202,119.6)}, rotate = 179.45] [color={rgb, 255:red, 0; green, 0; blue, 0 }  ][line width=0.75]    (10.93,-3.29) .. controls (6.95,-1.4) and (3.31,-0.3) .. (0,0) .. controls (3.31,0.3) and (6.95,1.4) .. (10.93,3.29)   ;
		\draw    (277,56.6) -- (320,56.6) ;
		\draw [shift={(320,56.6)}, rotate = 0] [color={rgb, 255:red, 0; green, 0; blue, 0 }  ][fill={rgb, 255:red, 0; green, 0; blue, 0 }  ][line width=0.75]      (0, 0) circle [x radius= 3.35, y radius= 3.35]   ;
		\draw [shift={(277,56.6)}, rotate = 0] [color={rgb, 255:red, 0; green, 0; blue, 0 }  ][fill={rgb, 255:red, 0; green, 0; blue, 0 }  ][line width=0.75]      (0, 0) circle [x radius= 3.35, y radius= 3.35]   ;
		\draw    (277,108) -- (320,108) ;
		\draw [shift={(320,108)}, rotate = 0] [color={rgb, 255:red, 0; green, 0; blue, 0 }  ][fill={rgb, 255:red, 0; green, 0; blue, 0 }  ][line width=0.75]      (0, 0) circle [x radius= 3.35, y radius= 3.35]   ;
		\draw [shift={(277,108)}, rotate = 0] [color={rgb, 255:red, 0; green, 0; blue, 0 }  ][fill={rgb, 255:red, 0; green, 0; blue, 0 }  ][line width=0.75]      (0, 0) circle [x radius= 3.35, y radius= 3.35]   ;
		\draw    (277,157.6) -- (320,157.6) ;
		\draw [shift={(320,157.6)}, rotate = 0] [color={rgb, 255:red, 0; green, 0; blue, 0 }  ][fill={rgb, 255:red, 0; green, 0; blue, 0 }  ][line width=0.75]      (0, 0) circle [x radius= 3.35, y radius= 3.35]   ;
		\draw [shift={(277,157.6)}, rotate = 0] [color={rgb, 255:red, 0; green, 0; blue, 0 }  ][fill={rgb, 255:red, 0; green, 0; blue, 0 }  ][line width=0.75]      (0, 0) circle [x radius= 3.35, y radius= 3.35]   ;
		\draw    (277,207.6) -- (320,207.6) ;
		\draw [shift={(320,207.6)}, rotate = 0] [color={rgb, 255:red, 0; green, 0; blue, 0 }  ][fill={rgb, 255:red, 0; green, 0; blue, 0 }  ][line width=0.75]      (0, 0) circle [x radius= 3.35, y radius= 3.35]   ;
		\draw [shift={(277,207.6)}, rotate = 0] [color={rgb, 255:red, 0; green, 0; blue, 0 }  ][fill={rgb, 255:red, 0; green, 0; blue, 0 }  ][line width=0.75]      (0, 0) circle [x radius= 3.35, y radius= 3.35]   ;

	\end{tikzpicture}
	
	\caption{Expanding a vertex of infinite degree into a $4$-ray.}
	\label{2kray}
\end{figure}

Consider the function $\rho : E(G) \to E(\widetilde{G})$ defined by
\[
\rho(uv) =
\begin{cases}
	uv,       & \text{if } u, v \in A, \\[0.2cm]
	u(v,n),   & \text{if } u \in A,\ v \in B \text{ and } uv = e_n^v, \\[0.2cm]
	(u,n)(v,m), & \text{if } u, v \in B \text{ and } uv = e_n^u = e_m^v.
\end{cases}
\]
We also define a vertex function
$\sigma :A\cup (B\times \mathbb{N})
\to V(G)$
by setting $\sigma((v,n)) = v$ for $v \in B$, and $\sigma(v) = v$ for $v \in A$.
Given a subgraph $H \subseteq \widetilde{G}$, we denote by $\rho^{-1}[H]$
the subgraph of $G$ whose vertex set consists of all $\sigma(v) \in V(G)$
such that $v \in A \cup (B\times \mathbb{N})$ is an endpoint of some edge in~$E(H)$,
and whose edge set is~$\rho^{-1}[E(H)]$.

We establish some properties of the expanded graph~$\widetilde{G}$.
	\begin{claim}
		\label{c1}
		$\widetilde{G}$ is locally finite.
	\end{claim}
	\begin{proof}
Let $v\in V(\widetilde{G})$. If $v$ is a vertex of a copy of $K_{2k}$ in a $2k$-ray $r_v$ of $\widetilde{G}$, then its degree is $2k+1$. If $v$ is a vertex that had finite degree in $G$, then its degree in $\widetilde{G}$ is the same as in $G$. Finally, suppose that $v$ is a vertex of a $2k$-ray $r_v$ of $\widetilde{G}$ that does not belong to any copy of $K_{2k}$. In this case, its degree is $2k+2k+1$ or $2k+1$.
	\end{proof}
	
	\begin{claim}
		\label{c2}
		If $G$ is connected, then $\widetilde{G}$ is connected.
	\end{claim}
	
	\begin{proof}
		
		It suffices to show that from a path $P$ between vertices $u$ and $v$ of $G$ we obtain a path $\widetilde{P}$ between some vertices in the sets  $\sigma^{-1}(\{u\})$ and $\sigma^{-1}(\{v\})$ of  $ V(\widetilde{G})$.

	Consider a path $v_0v_1\cdots v_{l+1}$ in $G$.  For each $i=0,\ldots,l$, we define a path $P_i$ in $\widetilde{G}$ as follows: $P_i$ is $(v_i,n)v_{i+1}$ if the next vertex $v_{i+1}$ has finite degree in $G$, $v_i$ has infinite degree in $G$ and $v_iv_{i+1}=e_n^{v_i}$ or $v_iv_{i+1}$ if $v_i$ and $v_{i+1}$ have finite degree; otherwise, $P_i$ is $e_i\in E(G)$ concatenated with a path in $r_{v_i}$ between $(v_i,r)$ and $(v_i,s)$, where $e_r^{\,v_{i}}=e_i$ and $e_s^{\,v_{i}}=e_{i+1}$. Notice that the concatenation of $P_0,\ldots,P_l$ is a path in $\widetilde{G}$. This construction is illustrated in Figure \ref{conca}.
	\end{proof}
	
\begin{figure}[htbp]
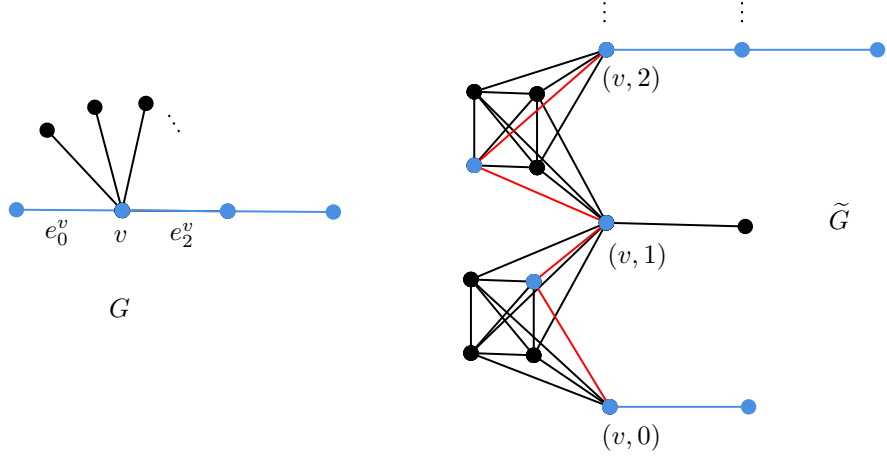

	\centering
	
	\tikzset{every picture/.style={line width=0.75pt}}
	

	
	\caption{Concatenation of paths in \(\widetilde{G}\) described in Claim~\ref{c2}.}
	\label{conca}
\end{figure}
	\begin{claim}
		\label{c3}
		Let $S\subseteq V(G)$ be a $2k$-edge-connected subset and $\widetilde{S}=\sigma^{-1}[S]$. Then $\widetilde{S}$ is a $2k$-edge-connected subset of $V(\widetilde{G})$.
	\end{claim}

	\begin{proof}

		Let $\widetilde{X} \subseteq E(\widetilde{G})$ be a set of fewer than $2k$ edges.
		Observe that $\rho^{-1}[\widetilde{X}]$ is an edge set of size less than $2k$.
		Since $S$ is $2k$-edge-connected, there exists a path between any two vertices
		of $S$ in $G - \rho^{-1}[\widetilde{X}]$.
		Therefore, by the proof of Claim~\ref{c2}, there exists a path between any two vertices
		of $\widetilde{S}$ in $\widetilde{G} - \widetilde{X}$.
		This proves that $\widetilde{S}$ is $2k$-edge-connected.
	\end{proof}

In the following results, we restrict our attention to finitely edge-separable graphs in order to avoid pathological phenomena that may arise when pairs of vertices are joined by infinitely many edge-disjoint paths (see Example~\ref{ex1}). Before proving Theorem~\ref{teoenum}, we establish a lemma concerning topological paths.

	\begin{lemma}[{\cite[Lemma 5.5]{guilherme}}]
		\label{lemapath}
		For an oriented graph $\vec{G}$ such that every two vertices $u$ and $v$ are finitely edge-connected, the following are equivalent:
		\begin{enumerate}
			\item There is an oriented topological path from $u$ to $v$;
			
			\item There is a continuous function
			\[
			\Phi : [0,1] \to \textsc{Etop}'(G)
			\]
			that respects edge orientation, such that $\Phi(0)=u$ and $\Phi(1)=v$; if, for two distinct points $x,y \in [0,1]$, $\Phi(x)=\Phi(y)$, then their image is either an end or a vertex that edge-dominates an end.
		\end{enumerate}
	\end{lemma}
		\begin{theorem}
			
			\label{teoenum}
		Let $G$ be a connected, finitely edge-separable, countable graph and $S \subseteq V(G)$ be a $2k$-edge-connected subset. 
		If Conjecture \ref{kriesell} holds, then $G$ with $\textsc{Etop}'(G) $ contains $k$ pairwise edge-disjoint  topological $S$-Steiner trees.
	\end{theorem}
	\begin{proof}
Let $\widetilde{G}$ be the expanded graph of $G$ and  $\widetilde{S}=\sigma^{-1}[S]$. By Claims \ref{c1}, \ref{c2} and \ref{c3}, $\widetilde{G}$ is a connected locally finite graph and $\widetilde{S}$ is $2k$-edge-connected in $\widetilde{G}$. Assuming that Conjecture~\ref{kriesell} is true, Theorem~\ref{localfin} gives $k$ pairwise edge-disjoint topological $\tilde{S}$-Steiner trees $\widetilde{H}_1,\widetilde{H}_2,\dots,\widetilde{H}_k$ of $\widetilde{G}$ with $\textsc{Etop}(\widetilde{G})$.

 We modify the topological trees
$
\widetilde{H}_1,\widetilde{H}_2,\ldots,\widetilde{H}_k
$ into subgraphs $\widetilde{T}_1,\widetilde{T}_2,\ldots,\widetilde{T}_k$
so that the subgraphs $$\rho^{-1}[\widetilde{T}_1], \rho^{-1}[\widetilde{T}_2],\cdots, \rho^{-1}[\widetilde{T}_k]$$ of $G$ are topological trees. Notice that a finite cycle can only arise in $\rho^{-1}[\widetilde{H}_i]$ if there exists a topological path in $\widetilde{H}_i$ between two vertices belonging to the same $2k$-ray $r_v$, while this topological path lies entirely outside $r_v$. On the other hand, an infinite topological cycle may arise from a ray \(r\) whose initial vertex dominates the edge-end \([r]_E\). If this occurs, we delete the first edge of \(r\). This destroys the topological cycle, since the graph is finitely edge-separable by hypothesis.

Let $1\leq i\leq k$, and let $P\subseteq \widetilde{H}_i$ be a topological path between two vertices $u,v\in V( \widetilde{H}_i)$ that belong to the same $2k$-ray $r_v$, such that
$
V(P)\cap V(r_v)=\{u,v\}.
$
Choose any edge $e\in E(P)$ and consider
$
P'=(P-e)
$. We perform this procedure for every topological path in $\widetilde{H}_i$ with the same property as $P$, and denote the resulting subgraph by $\widetilde{T}_i$. For each $1\leq i\leq k$ denote $T_i=\rho^{-1} [\widetilde{T}_i]$. 

\begin{claim}
	\label{caca}
The subgraphs $T_1,T_2,\cdots ,T_k$ are pairwise edge-disjoint topological trees.
\end{claim} 

\begin{proof}
We first show that $T_i$ is path-connected for each $1 \leq i \leq k$.
Let $u, v \in V(T_i)$, fix $\widetilde{u} \in \sigma^{-1}(\{u\}) \cap V(\widetilde{T}_i)$
and $\widetilde{v} \in \sigma^{-1}(\{v\}) \cap V(\widetilde{T}_i)$.
Denote by $P'$ the subgraph of $\widetilde{T}_i$ obtained from a topological
path $P$ between $\widetilde{u}$ and $\widetilde{v}$ in $\widetilde{H}_i$ by
removing the edges of those topological paths in $P$ that lie entirely outside
some expanded ray~$r_v$ and have both endpoints in~$r_v$.
Observe that the subgraph $\rho^{-1}[P']$ can be viewed as a continuous
function $\Phi : [0,1] \to ||G||$ satisfying $\Phi(0) = u$ and $\Phi(1) = v$,
and such that whenever $\Phi(t) = \Phi(t')$ for $t \neq t'$, their common
image is either an end or a vertex that edge-dominates an end,
as in Figure~\ref{domin}.

\begin{figure}[htbp]
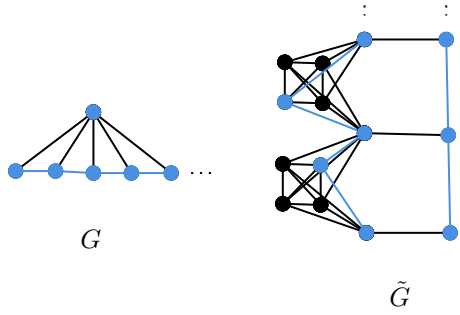

	\centering
	
	\tikzset{every picture/.style={line width=0.75pt}} 
	

	\caption{Obtaining a topological tree in $G$ containing a vertex edge-dominating an end from a topological tree in $\widetilde{G}$. }
	\label{domin}
\end{figure}

By Lemma~\ref{lemapath}, there exists a topological path between $u$ and $v$
in~$T_i$.
Note that $T_i$ contains no topological cycles, by the construction
of~$\widetilde{T}_i$.

We now prove that the topological trees \(T_i\) are pairwise edge-disjoint. Suppose, to the contrary, that there exist distinct indices \(i\) and \(j\) and an edge
$
e \in E(T_i)\cap E(T_j).
$
Since \(T_\ell=\rho^{-1}[\widetilde{T}_\ell]\) for every \(\ell\), the image of \(e\) under \(\rho\) satisfies
$
\rho(e)\in E(\widetilde{T}_i)\cap E(\widetilde{T}_j).
$
This contradicts the fact that \(\widetilde{T}_i\) and \(\widetilde{T}_j\) are edge-disjoint.
\end{proof}

Since each $\widetilde{T}_i$ contains $\widetilde{S} = \sigma^{-1}[S]$,
it follows that $S \subseteq V(T_i)$ for each $1 \leq i \leq k$.
Therefore, $G$ with $\textsc{Etop}'(G)$ contains $k$ pairwise
edge-disjoint topological $S$-Steiner trees.
		\end{proof}
		
		\begin{corollary}
			Let $G$ be a countable $2k$-edge-connected, finitely edge-separable graph. Then $G$ with $\textsc{Etop}'(G) $ contains $k$ pairwise edge-disjoint topological spanning trees.
		\end{corollary}
		
		\begin{proof}
			This follows from Theorems~\ref{nash} and \ref{teoenum} for $S=V(G)$.
		\end{proof}
		
Recall that in Lemma~\ref{limitrees} we showed that there exists a topological $S$-Steiner tree inside an $F$-limit of trees that contains $S$. We prove a converse in order to obtain a version of Kriesell’s conjecture in terms of $F$-limits of trees for every countable graph.

	\begin{proposition}
		\label{fintree}
		Let $G$ be a connected locally finite graph and $T\subseteq ||G||$ be a topological tree of $G$ with $\textsc{Etop}(G)$. Then $T$ is an $F$-limit of a sequence $\langle T_n\rangle_n$ of trees of $G$.
	\end{proposition}
\begin{proof}
	  Consider enumerations $E(T)=\{e_n\}_{n\in\mathbb{N}}$ and $E(G)\setminus E(T)=\{e'_n\}_{n\in\mathbb{N}}$. For each $n \in \mathbb{N}$, there exists a finite tree $T_n$ that contains the edges $e_0,e_1,\cdots,e_n$ and does not contain the edges $e'_0,e'_1,\cdots,e'_n$. 
	Indeed, such a tree exists because the connected components of $T$ are infinitely edge-connected, as $T$ is a topological tree. Indeed, fix vertices $u \in V(C)$ and $v \in V(C')$ belonging to distinct
	connected components $C$ and $C'$ of~$T$.
	Since $T$ is a topological tree, there exists a topological path
	in $T$ between $u$ and~$v$.
	Hence there exist rays $r \subseteq C$ and $r' \subseteq C'$ that are
	equivalent, with $u \in V( r)$ and $v \in V(r')$.
	This shows that $C$ and $C'$ cannot be separated by any finite set of edges.
	
	Fix a non-principal ultrafilter $F$ over $\mathbb{N}$. We claim that the $F$-limit of $\langle T_n\rangle_n$ is $T$. Denote by $H$ the $F$-limit of $\langle T_n\rangle_n$.

	 \begin{claim}
	 	\label{fca}
	 	$V(H)=V(T)$.
	 	\end{claim}
	 	
	 	\begin{proof}Consider $v\in V(T)$ and an edge $e_k\in E(T)$ such that $v$ is an endpoint of $e_k$. It follows that $v\in V(T_m)$ for all $m\geq k$. Hence $\{n\in\mathbb{N}:v\in V(T_n)\} \in F$. Therefore $v\in V(H)$. Conversely, let $v\in V(H)$. It follows that $\{n\in\mathbb{N}:v\in V(T_n)\}\in F$. Suppose that $v\notin V(T)$. Consider an increasing enumeration
	 		$e'_{n_0},e'_{n_1},\ldots,e'_{n_k}$
	 		of the edges incident with $v$. Notice that
	 		$v\notin V(T_m)$
	 		for every $m\geq n_k$, since $T_m$ contains no edge incident with $v$ and $T_m$ is connected. This contradicts the fact that
	 		$\{n\in\mathbb{N}:v\in V(T_n)\}$
	 		is infinite, because $F$ is a non-principal ultrafilter. Therefore,
	 		$v\in V(T)$.
\end{proof}
	 
	 \begin{claim}
	 	\label{fca1} $E(T)=E(H)$.
	 \end{claim}
	   
	 \begin{proof} Let $e_k\in E(T)$. Then $\{n\in\mathbb{N}: e_k\notin E(T_n)\}\notin F$, since $e_k\in E(T_m)$ for all $m\geq k$, and hence this set is finite. Therefore, $\{n\in\mathbb{N}: e_k\in E(T_n)\}\in F$, which implies that $e_k\in E(H)$.
	 Conversely, let $e\in E(H)$. Then $\{n\in\mathbb{N}: e\in E(T_n)\}\in F$. Suppose that $e\notin E(T)$. Then there exists $n_0\in\mathbb{N}$ such that $e\notin E(T_n)$ for all $n\geq n_0$, and thus the set $\{n\in\mathbb{N}: e\in E(T_n)\}$ is finite, a contradiction. Hence, $e\in E(T)$.
	 This proves that $E(T)=E(H)$.
	\end{proof}
	
	By Claims \ref{fca} and \ref{fca1}, it follows that $T$ is an $F$-limit of trees.
\end{proof}

\begin{theorem}
	\label{counta}
	Let $G$ be a countable connected graph and let $S\subseteq V(G)$ be a $2k$-edge-connected subset. If Conjecture \ref{kriesell} holds, then $G$ has $k$ pairwise edge-disjoint $F$-limits of trees which contain $S$.
\end{theorem}

\begin{proof}
Consider the expanded graph $\widetilde{G}$ of $G$ and denote $\widetilde{S}=\sigma^{-1}[S]$. We already know that $\widetilde{S}$ is $2k$-edge-connected by Claim \ref{c3}. Since $\widetilde{G}$ is locally finite and connected by Claims \ref{c1} and \ref{c2}, by Theorem~\ref{localfin} there exist $k$ pairwise edge-disjoint $\widetilde{S}$-Steiner topological trees $\widetilde{T}_1, \widetilde{T}_2, \cdots, \widetilde{T}_k$ in $\widetilde{G}$. By Proposition~\ref{fintree}, each $\widetilde{T}_i$, which contains $\widetilde{S}$, can be viewed as an $F$-limit of a sequence $\langle \widetilde{T^i_n} \rangle_n$ of finite trees.

Set $H_n^i=\rho^{-1}[\widetilde{T}_n^i]$. Notice that $H_n^i$ is a connected subgraph of $G$. Consider a spanning tree $T_n^i$ of $H_n^i$. Denote by $T_i$ the $F$-limit of $\langle T_n^i\rangle_n$. Since $\widetilde{T}_i$ contains $\widetilde{S}$, it follows that $T_i$ contains $S$, since we only consider spanning trees of $H_n^i$. 

Notice that $T_1, T_2, \ldots, T_k$ are pairwise edge-disjoint.
Suppose, for a contradiction, that there exists $e \in E(T_i) \cap E(T_j)$
for some $i \neq j$.
It follows that $\{n \in \mathbb{N} : e \in E(T_n^i)\} \in F$ and
$\{n \in \mathbb{N} : e \in E(T_n^j)\} \in F$.
In particular,
$\{n \in \mathbb{N} : \rho(e) \in E(\widetilde{T}_n^i)\} \in F$ and
$\{n \in \mathbb{N} : \rho(e) \in E(\widetilde{T}_n^j)\} \in F$.
Therefore $\rho(e) \in E(\widetilde{T}_i) \cap E(\widetilde{T}_j)$,
contradicting the fact that $\widetilde{T}_i$ and $\widetilde{T}_j$
are edge-disjoint.
\end{proof}

\begin{example}
	\label{ex1}

An $F$-limit of trees, as well as a topological tree, that contains a prescribed set of vertices may be trivial in graphs that are not finitely edge-separable, in the sense that it consists only of vertices.

Let $\langle P_n \rangle_{n\in\mathbb{N}}$ be an enumeration of the $u$–$v$ paths in the graph of Figure \ref{edge22}. Fix a non-principal ultrafilter $F$ on $\mathbb{N}$, and let $H$ be the $F$-limit of $\langle P_n \rangle_n$.

Since each edge $e$ appears in exactly one $P_n$, we have $\{n \in \mathbb{N} : e \in E(P_n)\} \notin F$, and hence $E(H)=\emptyset$. 

On the other hand, $\{n \in \mathbb{N} : u,v \in V(P_n)\}=\mathbb{N} \in F$, so $u,v \in V(H)$.
Finally, let $x \in \bigcup_{n\in\mathbb{N}} V(P_n) \setminus \{u,v\}$. Then $x \in V(P_k)$ for some $k$, and $x \notin V(P_m)$ for all $m \ge k$. Thus $\{n \in \mathbb{N} : x \in V(P_n)\} \notin F$, implying $x \notin V(H)$. Therefore, $V(H)=\{u,v\}$.
	\end{example}
	
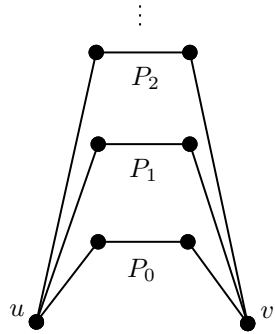
\begin{figure}[ht]
	\centering
		\tikzset{every picture/.style={line width=0.75pt}} 
		
		\begin{tikzpicture}[x=0.75pt,y=0.75pt,yscale=-1,xscale=1]
			
			\draw    (88,221.8) -- (119,181.6) ;
			\draw [shift={(119,181.6)}, rotate = 307.64] [color={rgb, 255:red, 0; green, 0; blue, 0 }  ][fill={rgb, 255:red, 0; green, 0; blue, 0 }  ][line width=0.75]      (0, 0) circle [x radius= 3.35, y radius= 3.35]   ;
			\draw [shift={(88,221.8)}, rotate = 307.64] [color={rgb, 255:red, 0; green, 0; blue, 0 }  ][fill={rgb, 255:red, 0; green, 0; blue, 0 }  ][line width=0.75]      (0, 0) circle [x radius= 3.35, y radius= 3.35]   ;
			\draw    (119,181.6) -- (164,181.6) ;
			\draw [shift={(164,181.6)}, rotate = 0] [color={rgb, 255:red, 0; green, 0; blue, 0 }  ][fill={rgb, 255:red, 0; green, 0; blue, 0 }  ][line width=0.75]      (0, 0) circle [x radius= 3.35, y radius= 3.35]   ;
			\draw [shift={(119,181.6)}, rotate = 0] [color={rgb, 255:red, 0; green, 0; blue, 0 }  ][fill={rgb, 255:red, 0; green, 0; blue, 0 }  ][line width=0.75]      (0, 0) circle [x radius= 3.35, y radius= 3.35]   ;
			\draw    (164,181.6) -- (194,222.6) ;
			\draw [shift={(194,222.6)}, rotate = 53.81] [color={rgb, 255:red, 0; green, 0; blue, 0 }  ][fill={rgb, 255:red, 0; green, 0; blue, 0 }  ][line width=0.75]      (0, 0) circle [x radius= 3.35, y radius= 3.35]   ;
			\draw [shift={(164,181.6)}, rotate = 53.81] [color={rgb, 255:red, 0; green, 0; blue, 0 }  ][fill={rgb, 255:red, 0; green, 0; blue, 0 }  ][line width=0.75]      (0, 0) circle [x radius= 3.35, y radius= 3.35]   ;
			\draw    (88,221.8) -- (119,132.6) ;
			\draw [shift={(119,132.6)}, rotate = 289.16] [color={rgb, 255:red, 0; green, 0; blue, 0 }  ][fill={rgb, 255:red, 0; green, 0; blue, 0 }  ][line width=0.75]      (0, 0) circle [x radius= 3.35, y radius= 3.35]   ;
			\draw [shift={(88,221.8)}, rotate = 289.16] [color={rgb, 255:red, 0; green, 0; blue, 0 }  ][fill={rgb, 255:red, 0; green, 0; blue, 0 }  ][line width=0.75]      (0, 0) circle [x radius= 3.35, y radius= 3.35]   ;
			\draw    (119,132.6) -- (165,132.6) ;
			\draw [shift={(165,132.6)}, rotate = 0] [color={rgb, 255:red, 0; green, 0; blue, 0 }  ][fill={rgb, 255:red, 0; green, 0; blue, 0 }  ][line width=0.75]      (0, 0) circle [x radius= 3.35, y radius= 3.35]   ;
			\draw [shift={(119,132.6)}, rotate = 0] [color={rgb, 255:red, 0; green, 0; blue, 0 }  ][fill={rgb, 255:red, 0; green, 0; blue, 0 }  ][line width=0.75]      (0, 0) circle [x radius= 3.35, y radius= 3.35]   ;
			\draw    (165,132.6) -- (194,222.6) ;
			\draw [shift={(194,222.6)}, rotate = 72.14] [color={rgb, 255:red, 0; green, 0; blue, 0 }  ][fill={rgb, 255:red, 0; green, 0; blue, 0 }  ][line width=0.75]      (0, 0) circle [x radius= 3.35, y radius= 3.35]   ;
			\draw [shift={(165,132.6)}, rotate = 72.14] [color={rgb, 255:red, 0; green, 0; blue, 0 }  ][fill={rgb, 255:red, 0; green, 0; blue, 0 }  ][line width=0.75]      (0, 0) circle [x radius= 3.35, y radius= 3.35]   ;
			\draw    (88,221.8) -- (118,86.6) ;
			\draw [shift={(118,86.6)}, rotate = 282.51] [color={rgb, 255:red, 0; green, 0; blue, 0 }  ][fill={rgb, 255:red, 0; green, 0; blue, 0 }  ][line width=0.75]      (0, 0) circle [x radius= 3.35, y radius= 3.35]   ;
			\draw [shift={(88,221.8)}, rotate = 282.51] [color={rgb, 255:red, 0; green, 0; blue, 0 }  ][fill={rgb, 255:red, 0; green, 0; blue, 0 }  ][line width=0.75]      (0, 0) circle [x radius= 3.35, y radius= 3.35]   ;
			\draw    (118,86.6) -- (165,86.6) ;
			\draw [shift={(165,86.6)}, rotate = 0] [color={rgb, 255:red, 0; green, 0; blue, 0 }  ][fill={rgb, 255:red, 0; green, 0; blue, 0 }  ][line width=0.75]      (0, 0) circle [x radius= 3.35, y radius= 3.35]   ;
			\draw [shift={(118,86.6)}, rotate = 0] [color={rgb, 255:red, 0; green, 0; blue, 0 }  ][fill={rgb, 255:red, 0; green, 0; blue, 0 }  ][line width=0.75]      (0, 0) circle [x radius= 3.35, y radius= 3.35]   ;
			\draw    (165,86.6) -- (194,222.6) ;
			\draw [shift={(194,222.6)}, rotate = 77.96] [color={rgb, 255:red, 0; green, 0; blue, 0 }  ][fill={rgb, 255:red, 0; green, 0; blue, 0 }  ][line width=0.75]      (0, 0) circle [x radius= 3.35, y radius= 3.35]   ;
			\draw [shift={(165,86.6)}, rotate = 77.96] [color={rgb, 255:red, 0; green, 0; blue, 0 }  ][fill={rgb, 255:red, 0; green, 0; blue, 0 }  ][line width=0.75]      (0, 0) circle [x radius= 3.35, y radius= 3.35]   ;
			\draw  [dash pattern={on 0.84pt off 2.51pt}]  (140,63) -- (140,75.6) ;
			
			\draw (73,211.4) node [anchor=north west][inner sep=0.75pt]    {$u$};
			\draw (199,212.4) node [anchor=north west][inner sep=0.75pt]    {$v$};
			\draw (132,188.4) node [anchor=north west][inner sep=0.75pt]    {$P_{0}$};
			\draw (133,138.4) node [anchor=north west][inner sep=0.75pt]    {$P_{1}$};
			\draw (134,92.4) node [anchor=north west][inner sep=0.75pt]    {$P_{2}$};

		\end{tikzpicture}
	
		\caption{The vertices $u$ and $v$ are the $F$-limit of the paths $P_n$.}
			\label{edge22}
	\end{figure}

We conclude this section by showing that, for uncountable graphs, the difficulty in finding edge-disjoint topological $S$-Steiner trees lies in the fact that $S$ may also be uncountable.

\begin{theorem}
	\label{teoteo}
	Let $G$ be a connected uncountable graph and $S\subseteq V(G)$  be a countable $2k$-edge-connected subset. If Conjecture \ref{kriesell} holds, then $G$ has $k$ pairwise edge-disjoint $F$-limits of finite trees which contain $S$.
\end{theorem}

\begin{proof}
Let $H$ be the subgraph obtained by taking, for each pair of vertices in $S$, a collection of $2k$ edge-disjoint paths between them. Then $S \subseteq V(H)$ and $S$ is $2k$-edge-connected in $H$. Moreover, for each pair of vertices in $S$ we select only finitely many paths, and since $S$ is countable, it follows that $H$ is countable. The result now follows by applying Theorem~\ref{counta} to $H$.
\end{proof}
\section{Rayless case}
\label{rayless}

We restrict our attention to rayless graphs. The idea is to decompose the rayless graph into countable
connected subgraphs in such a way that a subset of $S$ contained in a subgraph of the
decomposition remains $2k$-edge-connected in this subgraph. For this, we use bond-faithful decompositions.

\begin{definition}
	Let $G$ be a graph. A \emph{decomposition} of $G$ is a family $\{G_i : i \in I\}$ of subgraphs of $G$ such that every edge of $G$ belongs to exactly one $G_i$.
	A subgraph $H \subseteq G$ is called \emph{bond-faithful} if every finite bond of $H$ is a finite bond of $G$. 	A decomposition $\{G_i : i \in I\}$ is called:
	\begin{enumerate}
		\item \emph{bond-faithful} if:
		\begin{itemize}
			\item each $G_i$ is bond-faithful; and
			\item every finite cut of $G$ is contained in some $G_i$;
		\end{itemize}
		
		\item an \emph{$\alpha$-decomposition}, for an infinite cardinal $\alpha$, if $|V(G_i)| \leq \alpha$ for every $i \in I$.
	\end{enumerate}
\end{definition}

\begin{theorem}[{\cite[Theorem 3]{Laviolette2005}}]
	\label{laviolette}
	Every graph admits a bond-faithful decomposition into countable and connected subgraphs.
\end{theorem}

	\begin{proposition}[{\cite[Proposition 1]{Laviolette2005}}]
		\label{decom}
		If $H$ is a member of a bond-faithful $\alpha$-decomposition of $G$ and $x, y$ are any two vertices of $H$, then
		$$
		\gamma_H(x, y) = \min\{\alpha, \gamma_G(x, y)\}
		$$ where $\gamma_H(x,y)$ and $\gamma_G(x,y)$ denote the edge-connectivity between $x$ and $y$ in $H$ and $G$, respectively.
	\end{proposition}


 The proof for rayless graphs is lengthy and will be given in steps.  We begin by recalling the structure of rayless graphs.

\begin{definition}
	A graph $G$ has rank $0$ if it is finite.
	Let $\alpha > 0$ be an ordinal. A graph has rank $\alpha$ if it does not have a smaller rank and there is a finite subgraph $F \subseteq G$, called a \emph{reducer}, such that every connected component of $G - F$ has rank smaller than $\alpha$.
\end{definition}

\begin{theorem}[\cite{Schmidt1982}]
	A graph is rayless if and only if it has a rank.
\end{theorem}

\begin{theorem}
\label{final}
Let $G$ be a rayless connected graph and $S\subseteq V(G)$ be a $2k$-edge-connected subset. If Conjecture \ref{kriesell} holds, then $G$ contains $k$ pairwise edge-disjoint $F$-limits of trees which contain $S$. 
\end{theorem}
\begin{proof}
If $G$ has rank $0$, then the result follows from Conjecture \ref{kriesell}. If $S$ is countable, then the result follows from Theorem \ref{counta} when $G$ is countable and from Theorem \ref{teoteo} when $G$ is uncountable.  We may assume that $S$ contains all vertices of $G$ that are $2k$-edge-connected to at least one of its vertices (and therefore to all of them). Denote by $X\subseteq G $ a reducer of $G$ and by $\mathcal{D}=\{D_i\}_{i\in\mathcal{I}}$ an $\omega$-bond-faithful decomposition of $G$ given by Theorem \ref{laviolette}. By Proposition \ref{decom}, it follows that $S\cap V(D_i)$ is $2k$-edge-connected in $D_i$ for each $i\in\mathcal{I}$.

We will make some modifications to this $\omega$-bond-faithful decomposition. We prove some claims.

\begin{claim}
	\label{claim1}
	Let $G$ be a rayless graph of rank $1$ and $\mathcal{C}=\{C_j\}_{j\in\mathcal{J}}$ be the family of connected components of $G-X$, where $X$ is a reducer. We can assume that for every $j\in\mathcal{J}$ there is $i\in\mathcal{I}$ such that $C_j\subseteq D_i$ and $E(C_j,X)\subseteq D_i$, where $E(C_j,X)$ denotes the set of edges between $C_j$ and $X$. 
\end{claim} 

\begin{proof}

	Since $G$ has rank $1$, it follows that $C_j$ and $E(C_j,X)$ are finite. Hence there exist fragments $D_1,D_2,\cdots, D_m\in\mathcal{D}$ of the $\omega$-bond-faithful  decomposition $\mathcal{D}$ that contain at least one edge of $C_j$ or $E(C_j, X)$. Consider $\widetilde{D}=\bigcup_{1\leq n\le m} D_n$.  Then $\widetilde{D}$ is countable.  Define  $$\widetilde{\mathcal{D}}=(\mathcal{D}\setminus \{D_1,D_2,\cdots, D_m\})\cup \{\widetilde{D}\}$$
	
	Observe that, for a fragment \(D\in\mathcal{D}\), the set \(V(D)\cap S\) need not be \(2k\)-edge-connected within \(D\). Whenever this occurs, we merge \(D\) with countable fragments of \(\mathcal{D}\) containing the additional edges needed to make \(V(D)\cap S\) \(2k\)-edge-connected in the resulting subgraph.
\end{proof}

\begin{claim}
	\label{claim4.2}
	We may assume that every vertex of $S$ lies in a fragment of $\mathcal{D}$ that contains infinitely many vertices of $S$.
\end{claim}

\begin{proof}
	Let $\Gamma_0$ be the graph whose vertices are the fragments of $\mathcal D$, with two fragments adjacent whenever they have a common vertex. Since $G$ is connected, $\Gamma_0$ is connected. By Zorn's lemma, choose a maximal family $\{\mathcal A_\lambda:\lambda\in\Lambda\}$ of pairwise disjoint countable connected subsets of $V(\Gamma_0)$ such that the union of the fragments in each $\mathcal A_\lambda$ contains infinitely many vertices of $S$. Such subsets exist: choose countably many distinct vertices of $S$ and join fragments containing them by finite paths in $\Gamma_0$.
	
	Let $\mathcal R=\mathcal D\setminus\bigcup_{\lambda\in\Lambda}\mathcal A_\lambda$. By maximality, each component $C$ of $\Gamma_0[\mathcal R]$ contains only finitely many vertices of $S$ in the union of its fragments; otherwise, the same construction inside $C$ would give another member of the family. Every component $C$ containing a vertex of $S$ is adjacent to some $\mathcal A_\lambda$. For a fixed $\lambda$, only countably many such components are adjacent to $\mathcal A_\lambda$, since the union of the fragments in $\mathcal A_\lambda$ is countable and distinct components meet it in distinct vertices. For each such $C$, add to one adjacent $\mathcal A_\lambda$ finitely many fragments of $C$ which connect it to fragments containing all the vertices of $S$ occurring in $C$.
	
	The enlarged families are still pairwise disjoint, countable and connected. Merge the fragments in each of them and leave all remaining fragments unchanged. Every vertex of $S$ now lies in one of the merged fragments, and each such fragment contains infinitely many vertices of $S$. The resulting decomposition is still bond-faithful. Indeed, every finite cut of $G$ remains contained in a fragment. If $B$ is a finite bond of a merged fragment $H$ and $D$ is an old fragment containing an edge of $B$, then $B\cap E(D)$ is a nonempty finite cut of $D$ and contains a bond $B_D$ of $D$. Since $B_D$ is a bond of $G$, it is also a cut of $H$; as $B_D\subseteq B$ and $B$ is a bond of $H$, we have $B_D=B$.
\end{proof}

Write
$
\mathcal D_S=\{D\in\mathcal D:S\cap V(D)\text{ is infinite}\}.
$
By Claim~\ref{claim4.2}, the vertex sets of the fragments in $\mathcal D_S$ cover $S$.

\begin{claim}
	\label{clsep}
	If a vertex $v\in V(G)$ cannot be separated from $S$ by a finite set of edges, then $v\in S$.
\end{claim}

\begin{proof}
	Let $E\subseteq E(G)$ have fewer than $2k$ edges. By hypothesis, there exists $u\in S$ joined to $v$ in $G-E$. Since $S$ is $2k$-edge-connected, all vertices of $S$ lie in the same component of $G-E$. Thus, $v$ is $2k$-edge-connected to every vertex of $S$. By the choice of the enlarged set $S$, we have $v\in S$.
\end{proof}

\begin{claim}
	\label{clred}
	If $S$ intersects infinitely many components of $G-X$, then $X$ contains a vertex of $S$.
\end{claim}

\begin{proof}
	For each component of $G-X$ that intersects $S$, choose one vertex of $S$ contained in that component. Denote the resulting set of chosen vertices by $S'$. By hypothesis, $S'$ is infinite. Since $X$ is finite, there exists a vertex $x\in X$ that cannot be separated from $S'$ by any finite set of edges. Indeed, otherwise the union of finite sets separating each vertex of $X$ from $S'$ would be finite, but some component of $G-X$ containing a vertex of $S'$ would avoid this finite set and would still be joined to $X$.

	Since $S'\subseteq S$, the vertex $x$ cannot be separated from $S$ by a finite set of edges. Claim~\ref{clsep} gives $x\in S$. Hence $x\in X\cap S$.
\end{proof}

\textbf{Rank 1 case:} Suppose that $G$ is a rayless graph of rank $1$ and $S$ is uncountable. Then infinitely many components of $G-X$ are intersected by $S$. By Claim~\ref{clred}, the reducer $X$ contains a vertex of $S$. We show that we may assume that some fragment of the decomposition $\mathcal{D}$ contains $V(X)\cap S$. Consider $V(X)\cap S = \{v_0, v_1, \dots, v_m\}$ and denote by $D_0$ the fragment of the decomposition such that $v_0 \in V(D_0)$. For each $1 \le n \le m$, since $G$ is connected,  there is a path $P_n$ from $v_0$ to $v_n$. Since each $P_n$ is finite, it follows that there are fragments of the decomposition $D_0, D_1, \dots, D_m\in\mathcal{D}$ such that
$
\bigcup_{n=1}^m E(P_n) \subseteq \bigcup_{j=0}^m E( D_j )
$.
Consider $\tilde D = D_0 \cup D_1 \cup \dots \cup D_m$ and  define a new decomposition  by 
$
\mathcal{D}' = (\mathcal{D} \setminus \{D_0,\dots,D_m\}) \cup \{\tilde D\}
$.
Notice that it may happen that a subset of $S$ in a fragment of this new decomposition is not $2k$-edge-connected in the fragment. If this occurs, we further merge fragments until the required $2k$-edge-connectivity is achieved, as in the proof of Claim \ref{claim1}. 

After this modification, we keep the notation $\mathcal D_S$ for the fragments that contain infinitely many vertices of $S$; their vertex sets still cover $S$. Since $G$ has rank $1$, every component of $G-X$ is finite. Hence every $D_i\in\mathcal D_S$ intersects infinitely many components of $G-X$ in vertices of $S$. The same argument used in Claim~\ref{clred}, now inside $D_i$ and with the finite set $V(D_i)\cap X$, gives
$
V(D_i)\cap X\cap S\neq\varnothing.
$
Therefore, for every $D_i\in\mathcal D_S$,
$
V(D_i) \cap V(\tilde{D}) \cap X\cap S\neq  \varnothing.
$
Let $\mathcal D^*=\mathcal D_S\cup\{\tilde D\}$. For each $D_i\in\mathcal D^*$ let $H_1^i,H_2^i,\cdots, H_k^i$ be $k$ pairwise edge-disjoint $F$-limits of finite trees in $D_i$ given by Theorem \ref{counta}. For each \(1\leq j\leq k\) and $D_i\in\mathcal D^*$, let \(H_j^i\) be the \(F\)-limit of the sequence
$
\bigl\langle T_\ell^{(j,i)}\bigr\rangle_{\ell\in\mathbb{N}}.
$
For every $D_i\in\mathcal D_S$, choose $v_0^i\in V(D_i)\cap V(\tilde D)\cap X\cap S$ and fix an enumeration
\[
S_i=S\cap V(D_i)=\{v_m^i:m\in\mathbb{N}\}
\]
For $\tilde D$, fix an enumeration of $S\cap V(\tilde D)$. Define
$
S_{i,n}=\{v_m^i:m\leq n\}
$
for $D_i\in\mathcal D_S$, while for $D_i=\tilde D$ let $S_{i,n}$ also contain the finite set $X\cap S$. Since \(S_{i,n}\) is finite, there exists a tree \(T_{\ell(i,j,n)}^{(j,i)}\) in the sequence \(\langle T_\ell^{(j,i)}\rangle_{\ell\in\mathbb{N}}\) such that
$
S_{i,n}\subseteq V\bigl(T_{\ell(i,j,n)}^{(j,i)}\bigr).
$

For fixed \(n\in\mathbb{N}\) and \(1\leq j\leq k\), consider the union
\[
U_n^j
=
\bigcup_{D_i\in\mathcal D^*}
T_{\ell(i,j,n)}^{(j,i)}
\]

This union is a connected subgraph, since the tree contained in $\tilde D$ contains $X\cap S$ and every other tree contains its chosen vertex $v_0^i\in V(D_i)\cap V(\tilde D)\cap X\cap S$. Let \(\widetilde{T}_n^j\) be a spanning tree of \(U_n^j\). Since a spanning tree has the same vertex set as the subgraph it spans, we conclude that
$
v_m^i\in V(\widetilde{T}_n^j)
$
for every $D_i\in\mathcal D_S$ and \(m\leq n\) (see Figure \ref{limtree}). Denote by $\widetilde{H}_j$ the $F$-limit of $\langle \widetilde{T}_n^j\rangle_n$. Since the fragments form an edge decomposition and the $F$-limits inside each fragment are pairwise edge-disjoint, it follows that $\widetilde{H}_1, \widetilde{H}_2, \cdots, \widetilde{H}_k$ are also pairwise edge-disjoint and contain $S$ by construction.

\begin{figure}[htbp]
	\centering
	
	\definecolor{treered}{RGB}{220,20,45}
	
	\begin{tikzpicture}[
		x=1cm,
		y=1cm,
		fragment/.style={
			draw=black,
			line width=0.6pt,
			line cap=round,
			line join=round
		},
		tree/.style={
			draw=treered,
			line width=1.3pt,
			line cap=round,
			line join=round
		},
		vertex/.style={
			circle,
			fill=black,
			inner sep=2.5pt
		}
		]
		
		\draw[fragment]
		(3.20,4.00)
		.. controls (3.75,4.45) and (5.10,4.30) .. (6.05,4.15)
		.. controls (7.10,4.35) and (8.35,4.45) .. (8.65,3.65)
		.. controls (8.90,2.90) and (8.45,1.25) .. (7.65,0.75)
		.. controls (6.80,0.20) and (5.15,0.05) .. (4.55,0.80)
		.. controls (3.85,1.65) and (3.05,3.20) .. cycle;
		
		\draw[fragment]
		(4.05,3.75)
		.. controls (3.40,3.75) and (2.75,4.10) .. (2.25,4.75)
		.. controls (1.85,5.30) and (1.15,6.10) .. (1.60,6.55)
		.. controls (2.25,7.15) and (3.60,6.70) .. (4.20,6.00)
		.. controls (4.65,5.45) and (5.25,4.35) .. cycle;
		
		\draw[fragment]
		(5.95,3.90)
		.. controls (5.35,4.15) and (4.75,5.15) .. (5.10,5.90)
		.. controls (5.45,6.70) and (5.75,7.35) .. (6.55,7.35)
		.. controls (7.20,7.35) and (7.35,6.50) .. (7.85,5.95)
		.. controls (8.15,5.55) and (7.70,4.70) .. (7.15,4.05)
		.. controls (6.80,3.65) and (6.25,3.55) .. cycle;
		
		\draw[fragment]
		(7.75,3.90)
		.. controls (7.20,4.15) and (6.65,4.85) .. (6.95,5.60)
		.. controls (7.25,6.35) and (7.70,6.85) .. (8.45,7.30)
		.. controls (9.10,7.70) and (9.30,6.45) .. (9.75,6.10)
		.. controls (10.25,5.70) and (9.70,4.60) .. (9.15,4.10)
		.. controls (8.75,3.75) and (8.20,3.65) .. cycle;
		
		\node at (2.15,7.25) {$D_i$};
		\node at (6.15,7.85) {$D_j$};
		\node at (9.55,7.55) {$D_\ell$};
		
		\draw[tree]
		(4.05,3.90) -- (3.05,4.55) -- (2.45,4.95);
		
		\draw[tree]
		(3.05,4.55) -- (2.80,5.35);
		
		\draw[tree]
		(3.05,4.55) -- (3.70,5.15) -- (3.35,5.55);
		
		\draw[tree]
		(6.05,4.15) -- (6.05,4.85) -- (5.60,5.55);
		
		\draw[tree]
		(6.05,4.85) -- (6.45,5.45);
		
		\draw[tree]
		(7.75,4.15) -- (7.65,5.35) -- (8.45,5.20);
		
		\draw[tree]
		(8.45,5.20) -- (8.05,5.75);
		
		\draw[tree]
		(8.45,5.20) -- (8.85,5.65);
		
		\coordinate (c) at (6.15,2.85);
		
		\draw[tree] (4.05,3.90) -- (c);
		\draw[tree] (6.05,4.15) -- (c);
		\draw[tree] (7.75,4.15) -- (c);
		\draw[tree] (c) -- (6.15,1.80);
		
		\node[vertex] at (4.05,3.90) {};
		\node[vertex] at (6.05,4.15) {};
		\node[vertex] at (7.75,4.15) {};
		\node[vertex,fill=treered] at (6.15,1.80) {};
		
		\node at (3.90,3.50) {$v_0^i$};
		\node at (6.45,3.95) {$v_0^j$};
		\node at (8.09,3.93) {$v_0^\ell$};
		\node at (4.75,2.10) {$\widetilde{D}$};
		\node at (6.15,1.43) {$v_0$};
		
		\node[text=treered] at (7.35,2.15)
		{$\widetilde{T}_n^j$};
		
	\end{tikzpicture}
	
	\caption{The tree \(\widetilde{T}_n^j\) obtained from the union of
		trees contained in the fragments of the decomposition.}
	\label{limtree}
\end{figure}
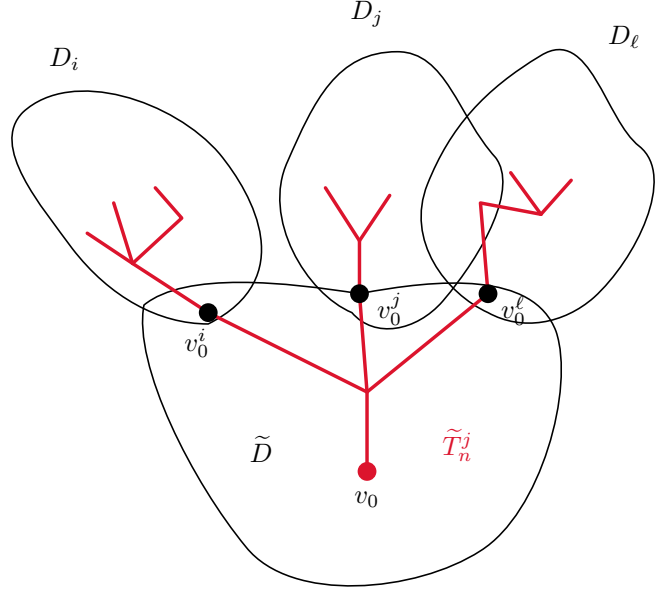

\textbf{Rank 2 case:} To clarify the arguments that will be used in the
general case, we also present the proof for rayless graphs of rank $2$.
Suppose that $G$ is a rayless graph of rank $2$ and let $X$ be a
reducer of $G$. Then every connected component of $G-X$ has rank at
most $1$.

Consider a connected component $C$ of $G-X$. Since $C$ has rank at
most $1$, it has a finite reducer $R_C$. If $C$ contains infinitely
many vertices of $S$, then, since $R_C$ is finite and every component
of $C-R_C$ is finite, the set $S$ intersects infinitely many components
of $C-R_C$. As in the proof of Claim~\ref{clred}, there exists a vertex $x\in R_C$ that cannot be separated from $S\cap V(C)$ by any finite set of edges of $C$. Hence $x$ cannot be separated from $S$ by any finite set of edges of $G$. By Claim~\ref{clsep}, we have $x\in S$. Thus, every reducer of a component
of $G-X$ containing infinitely many vertices of $S$ contains a vertex
of $S$.

Now consider a connected component $C$ of $G-X$ containing only
finitely many vertices of $S$, say $S\cap V(C)=\{u_0,u_1,\ldots,u_l\}$.
For each $i\in\{0,\ldots,l\}$, let $D_i$ be a fragment of the
decomposition $\mathcal D$ containing $u_i$. Since $C$ is connected,
for each $i$ there is a finite path in $C$ joining $u_0$ to $u_i$.
Hence, there are fragments $D_0,D_1,\ldots,D_m\in\mathcal D$ containing
all the edges of these paths. As in the proof of Claim~\ref{claim1}, we merge
these fragments and obtain a countable fragment $\widetilde D$ such
that all vertices of $S\cap V(C)$ belong to $\widetilde D$. We may
then modify the decomposition by replacing $D_0,\ldots,D_m$ with
$\widetilde D$. If necessary, we further merge fragments so that the
set of vertices of $S$ contained in each fragment is $2k$-edge-connected
inside that fragment, as in the proof of Claim~\ref{claim1}.

By Claim~\ref{claim4.2}, we may assume that every vertex of $S$ lies in a fragment of $\mathcal D$ containing infinitely many vertices of $S$. Moreover, for
every reducer $R$ of $G$ or of a connected component of $G-X$ which
contains vertices of $S$, all vertices of $S\cap V(R)$ are contained
in a fragment of $\mathcal D$. Denote such a fragment by
$D_R$.

For each fragment $D\in\mathcal D$, let $S_D=S\cap V(D)$. By
Theorem~\ref{counta}, there exist $k$ pairwise edge-disjoint $F$-limits of
finite trees in $D$, say $H_D^1,H_D^2,\ldots,H_D^k$, such that
$S_D\subseteq V(H_D^j)$ for every $1\leq j\leq k$.

The fragments need not all contain a vertex of $S\cap X$. However,
a fragment containing a vertex of $S$ that lies in a reducer $R_C$ of
a component $C$ of $G-X$ can be joined to the fragment $D_X$ associated
with $X$ through the fragment $D_{R_C}$ associated with the reducer
$R_C$. Thus, starting from any fragment containing vertices of $S$ and
following the corresponding reducers, we obtain a chain of fragments
leading to a fragment associated with $X$ whenever $S\cap V(X)\neq
\varnothing$.

Fix $n\in\mathbb N$ and $1\leq j\leq k$. For every fragment
$D\in\mathcal D$, enumerate $S_D=\{v_m^D:m\in\mathbb N\}$ whenever
$S_D$ is infinite, and define $S_{D,n}=\{v_m^D:m\leq n\}$. If $S_D$ is finite, put $S_{D,n}=S_D$. Since
$S_{D,n}$ is finite, there exists a tree $T_{D,\ell(D,j,n)}^j$ in the
sequence defining $H_D^j$ such that $S_{D,n}\subseteq
V(T_{D,\ell(D,j,n)}^j)$.

For fixed $n\in\mathbb N$ and $1\leq j\leq k$, take the union of all
trees $T_{D,\ell(D,j,n)}^j$ and of the finite connecting paths
corresponding to the chains of fragments described above. Denote this
union by $U_n^j$. By construction, $U_n^j$ is connected and contains
the first $n$ vertices of $S$ lying in every fragment. Let
$\widetilde T_n^j$ be a spanning tree of $U_n^j$.

Since a spanning tree has the same vertex set as the subgraph it
spans, every vertex of $S$ eventually belongs to
$V(\widetilde T_n^j)$. Let $\widetilde H^j$ be the $F$-limit of
$\langle\widetilde T_n^j\rangle_{n\in\mathbb N}$. It follows that
$S\subseteq V(\widetilde H^j)$.

Finally, the trees used inside distinct fragments are pairwise
edge-disjoint for different values of $j$, and the connecting paths
can also be chosen inside the corresponding trees. Therefore
$\widetilde H^1,\widetilde H^2,\ldots,\widetilde H^k$ are pairwise
edge-disjoint. Hence $G$ contains $k$ pairwise edge-disjoint
$F$-limits of trees which contain $S$.

\textbf{General case:} We construct finite chains of fragments joining each fragment of $\mathcal D_S$ to a fixed fragment containing $V(X)\cap S$. Consecutive fragments in each chain will have a common vertex in $S$. These chains will allow us to join the trees obtained inside the fragments.

Let $G$ be a rayless graph of rank $\alpha$ and let $X$ be a reducer of $G$. By enlarging $X$ as described below, we may assume that $S$ intersects infinitely many components of $G-X$. Let $\mathcal{D}$ be an $\omega$-bond-faithful decomposition of $G$ given by Theorem \ref{laviolette}. For each connected  component $C_1^i$ of $G - X$, denote by $X_1^i$ its reducer. Similarly, for each component $C_2^i$ of $(G - X) - \bigcup_j X_1^j$, denote by $X_2^i$ its reducer. Continuing this process for an ordinal $\kappa < \alpha$ and for each connected component $C_\kappa^i$ of $
(G - X) - \bigcup_j \left( \bigcup_{0 < \ell < \kappa} X_\ell^j \right)
$
we denote by $X_\kappa^i$ its reducer (see Figure \ref{comprayless}).

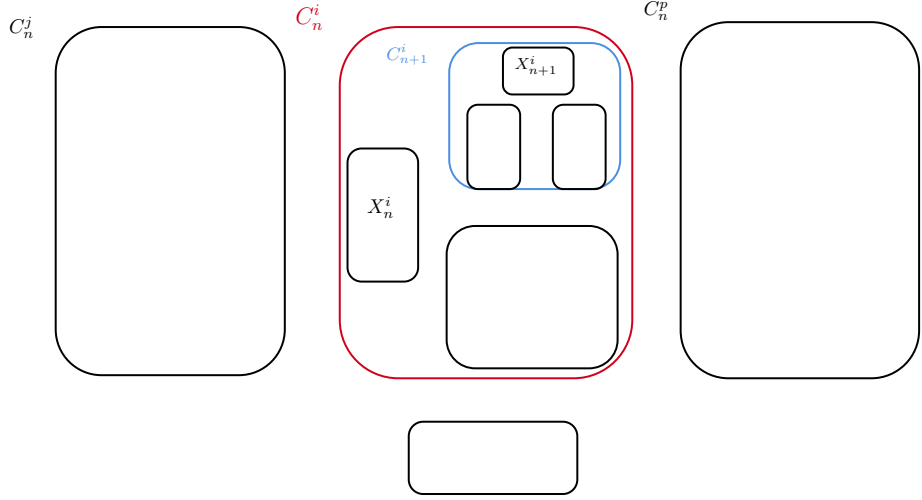
\begin{figure}[ht]
	
	\centering
	\tikzset{every picture/.style={line width=0.75pt}} 
	
	\begin{tikzpicture}[x=0.75pt,y=0.75pt,yscale=-1,xscale=1]
		
		\draw  [color={rgb, 255:red, 208; green, 2; blue, 27 }  ,draw opacity=1 ] (253.68,48.09) .. controls (253.68,31.88) and (266.82,18.74) .. (283.03,18.74) -- (371.09,18.74) .. controls (387.3,18.74) and (400.44,31.88) .. (400.44,48.09) -- (400.44,165.59) .. controls (400.44,181.8) and (387.3,194.94) .. (371.09,194.94) -- (283.03,194.94) .. controls (266.82,194.94) and (253.68,181.8) .. (253.68,165.59) -- cycle ;
		\draw   (288.29,223.99) .. controls (288.29,219.98) and (291.53,216.73) .. (295.54,216.73) -- (365.61,216.73) .. controls (369.61,216.73) and (372.86,219.98) .. (372.86,223.99) -- (372.86,245.75) .. controls (372.86,249.75) and (369.61,253) .. (365.61,253) -- (295.54,253) .. controls (291.53,253) and (288.29,249.75) .. (288.29,245.75) -- cycle ;
		\draw   (424.7,40.22) .. controls (424.7,27) and (435.42,16.29) .. (448.63,16.29) -- (520.41,16.29) .. controls (533.62,16.29) and (544.33,27) .. (544.33,40.22) -- (544.33,171.02) .. controls (544.33,184.23) and (533.62,194.94) .. (520.41,194.94) -- (448.63,194.94) .. controls (435.42,194.94) and (424.7,184.23) .. (424.7,171.02) -- cycle ;
		\draw   (111.17,41.73) .. controls (111.17,29.03) and (121.46,18.74) .. (134.16,18.74) -- (203.12,18.74) .. controls (215.81,18.74) and (226.11,29.03) .. (226.11,41.73) -- (226.11,170.4) .. controls (226.11,183.1) and (215.81,193.39) .. (203.12,193.39) -- (134.16,193.39) .. controls (121.46,193.39) and (111.17,183.1) .. (111.17,170.4) -- cycle ;
		\draw   (257.57,86.76) .. controls (257.57,82.84) and (260.74,79.67) .. (264.66,79.67) -- (285.91,79.67) .. controls (289.83,79.67) and (293,82.84) .. (293,86.76) -- (293,139.36) .. controls (293,143.27) and (289.83,146.44) .. (285.91,146.44) -- (264.66,146.44) .. controls (260.74,146.44) and (257.57,143.27) .. (257.57,139.36) -- cycle ;
		\draw  [color={rgb, 255:red, 74; green, 144; blue, 226 }  ,draw opacity=1 ] (308.6,41.43) .. controls (308.6,33.34) and (315.16,26.78) .. (323.25,26.78) -- (379.72,26.78) .. controls (387.81,26.78) and (394.37,33.34) .. (394.37,41.43) -- (394.37,85.38) .. controls (394.37,93.47) and (387.81,100.03) .. (379.72,100.03) -- (323.25,100.03) .. controls (315.16,100.03) and (308.6,93.47) .. (308.6,85.38) -- cycle ;
		\draw   (307.3,132.83) .. controls (307.3,124.94) and (313.69,118.54) .. (321.58,118.54) -- (378.79,118.54) .. controls (386.68,118.54) and (393.07,124.94) .. (393.07,132.83) -- (393.07,175.67) .. controls (393.07,183.56) and (386.68,189.95) .. (378.79,189.95) -- (321.58,189.95) .. controls (313.69,189.95) and (307.3,183.56) .. (307.3,175.67) -- cycle ;
		\draw   (335.55,33.73) .. controls (335.55,31.13) and (337.66,29.03) .. (340.26,29.03) -- (366.27,29.03) .. controls (368.87,29.03) and (370.98,31.13) .. (370.98,33.73) -- (370.98,47.85) .. controls (370.98,50.45) and (368.87,52.56) .. (366.27,52.56) -- (340.26,52.56) .. controls (337.66,52.56) and (335.55,50.45) .. (335.55,47.85) -- cycle ;
		\draw   (360.58,63.01) .. controls (360.58,60.09) and (362.95,57.72) .. (365.88,57.72) -- (381.76,57.72) .. controls (384.69,57.72) and (387.06,60.09) .. (387.06,63.01) -- (387.06,94.74) .. controls (387.06,97.66) and (384.69,100.03) .. (381.76,100.03) -- (365.88,100.03) .. controls (362.95,100.03) and (360.58,97.66) .. (360.58,94.74) -- cycle ;
		\draw   (317.7,63.01) .. controls (317.7,60.09) and (320.07,57.72) .. (323,57.72) -- (338.88,57.72) .. controls (341.81,57.72) and (344.18,60.09) .. (344.18,63.01) -- (344.18,94.74) .. controls (344.18,97.66) and (341.81,100.03) .. (338.88,100.03) -- (323,100.03) .. controls (320.07,100.03) and (317.7,97.66) .. (317.7,94.74) -- cycle ;
		
		\draw (266.25,102.66) node [anchor=north west][inner sep=0.75pt]  [xscale=0.75,yscale=0.75]  {$X_{n}^{i}$};
		\draw (340.38,32.04) node [anchor=north west][inner sep=0.75pt]  [font=\footnotesize,xscale=0.75,yscale=0.75]  {$X_{n+1}^{i}$};
		\draw (230.34,6.27) node [anchor=north west][inner sep=0.75pt]  [font=\large,color={rgb, 255:red, 208; green, 2; blue, 27 }  ,opacity=1 ,xscale=0.75,yscale=0.75]  {$C_{n}^{i}$};
		\draw (276.02,26.41) node [anchor=north west][inner sep=0.75pt]  [font=\small,color={rgb, 255:red, 74; green, 144; blue, 226 }  ,opacity=1 ,xscale=0.75,yscale=0.75]  {$C_{n+1}^{i}$};
		\draw (86.69,12.52) node [anchor=north west][inner sep=0.75pt]  [xscale=0.75,yscale=0.75]  {$C_{n}^{j}$};
		\draw (405.1,4.74) node [anchor=north west][inner sep=0.75pt]  [xscale=0.75,yscale=0.75]  {$C_{n}^{p}$};

	\end{tikzpicture}

	\caption{The structure of $C_n^i$ and $X_n^i$ in the rayless graph.}
	\label{comprayless}
\end{figure}

Recall that $\mathcal D_S$ denotes the fragments containing infinitely many vertices of $S$, and that their vertex sets cover $S$ by Claim \ref{claim4.2}. 
For a reducer $ X_\kappa^i \subseteq C_\kappa^i$  we may further assume that either $ C_\kappa^i - X_\kappa^i $ contains only finitely many vertices of $ S$, or that 
$S $ intersects infinitely many components of $ C_\kappa^i - X_\kappa^i$. Indeed, if $C_\kappa^i - X_\kappa^i$ has only finitely many components intersecting $S$ and contains infinitely many vertices of $S$, take the union of the reducers of these components with $X_\kappa^i$ and consider this union as a new reducer. In this process, the number of connected components of $C_{\kappa}^i-X_{\kappa}^i$ increases. Since the connected components have a rank, this process must terminate after finitely many steps. Consequently, in the resulting decomposition, infinitely many connected components of $C_{\kappa}^i-X_{\kappa}^i$ are intersected by $S$. Moreover, if $C_\kappa^i-X_\kappa^i$ has infinitely many components that intersect $S$, then, as in the proof of Claim~\ref{clred}, some vertex of $X_\kappa^i$ cannot be separated from $S\cap V(C_\kappa^i)$ by any finite set of edges of $C_\kappa^i$. Claim~\ref{clsep} therefore gives $X_\kappa^i\cap S\neq\emptyset$.
If $C_\kappa^i - X_\kappa^i$ contains only finitely many vertices of $S$, we may assume that there exists $D\in\mathcal{D}$ containing all these vertices.

Let $D\in\mathcal D_S$. There exists a reducer $X_\ell^j\subseteq C_\ell^j$ such that infinitely many components of $C_\ell^j-X_\ell^j$ contain vertices of $S\cap V(D)$. Otherwise, at each step, only finitely many connected components would meet $S\cap V(D)$, and removing their reducers would strictly decrease the largest rank among these components. Since there is no infinite strictly decreasing sequence of ordinals, this process would terminate after finitely many steps. As each reducer is finite, this would imply that $S\cap V(D)$ is finite, a contradiction. Applying the argument from Claim~\ref{clred} inside the connected fragment $D$, followed by Claim~\ref{clsep}, gives a reducer meeting $V(D)\cap S$. Choose such a reducer $X_\kappa^i$ with the smallest possible index $\kappa$. Thus,
$
X_\kappa^i\cap V(D)\cap S\neq\emptyset.
$

Consider a fragment $D_{n_0} \in \mathcal{D}$ whose vertex set contains $V(X) \cap S$.	For each $D \in \mathcal{D}$ that intersects infinitely many connected components of $G-X$ in vertices of $ S$ we obtain that
$
V(D) \cap V(D_{n_0}) \cap V(X)\cap S \neq \emptyset
$. Let  $X_\kappa^i \subseteq C_\kappa^i$ be a reducer such that $C_\kappa^i - X_\kappa^i$ has infinitely many connected components intersecting $S$ and there exists $D \in \mathcal{D}$ for which $X_\kappa^i$ is the reducer of smallest index $\kappa$ satisfying $V(D) \cap X_\kappa^i \cap S \neq \emptyset$. We may assume that there exists $D_\kappa^i\in\mathcal{D}$ such that
$
X_\kappa^i \cap S \subseteq V(D_\kappa^i)
$.

Let $D\in\mathcal{D}_S$. There exists a reducer $X_\kappa^i\subseteq C_\kappa^i$ such that $X_\kappa^i$ is the reducer of smallest index $\kappa$ such that $V(D)\cap X_\kappa^i\cap S\neq \emptyset$. Consider $D_0=D_\kappa^i\in\mathcal{D}$ that contains $X_\kappa^i\cap S$ in its vertex set. Consider $D_1 \in \mathcal{D}$ such that
$
V(D_1) \cap X_\kappa^i \cap S \neq \emptyset$
and $
V(D_1) \cap X_{\kappa_1}^j \cap S \neq \emptyset
$
for some $\kappa_1 < \kappa$. Thus, we obtain $D_2 \in \mathcal{D}$ and $\kappa_2< \kappa_1$ such that
$D_2$ contains $X_{\kappa_1}^j \cap S$ in its vertex set and $V(D_2) \cap X_{\kappa_2}^j \cap S \neq \emptyset
$ (see Figure \ref{fig2}). Continuing this process, we obtain a strictly decreasing sequence of ordinals
$
\kappa > \kappa_1 > \kappa_2 > \cdots > \kappa_n > \cdots
$
which must be finite, that is, 
$
\kappa > \kappa_1 > \cdots > \kappa_m = 0
$
for some natural number $m\in\mathbb{N}$ such that $V(D_i)\cap V(D_{i-1})\cap S\neq\emptyset$ for $1\leq i\leq m$.

\begin{figure}[ht]
	
	\centering
	\tikzset{every picture/.style={line width=0.75pt}} 
	
	\begin{tikzpicture}[x=0.75pt,y=0.75pt,yscale=-1,xscale=1]
		
		\draw  [color={rgb, 255:red, 0; green, 0; blue, 0 }  ,draw opacity=1 ] (246.68,64.09) .. controls (246.68,47.88) and (259.82,34.74) .. (276.03,34.74) -- (364.09,34.74) .. controls (380.3,34.74) and (393.44,47.88) .. (393.44,64.09) -- (393.44,181.59) .. controls (393.44,197.8) and (380.3,210.94) .. (364.09,210.94) -- (276.03,210.94) .. controls (259.82,210.94) and (246.68,197.8) .. (246.68,181.59) -- cycle ;
		\draw   (281.29,239.99) .. controls (281.29,235.98) and (284.53,232.73) .. (288.54,232.73) -- (358.61,232.73) .. controls (362.61,232.73) and (365.86,235.98) .. (365.86,239.99) -- (365.86,261.75) .. controls (365.86,265.75) and (362.61,269) .. (358.61,269) -- (288.54,269) .. controls (284.53,269) and (281.29,265.75) .. (281.29,261.75) -- cycle ;
		\draw   (104.17,57.73) .. controls (104.17,45.03) and (114.46,34.74) .. (127.16,34.74) -- (196.12,34.74) .. controls (208.81,34.74) and (219.11,45.03) .. (219.11,57.73) -- (219.11,186.4) .. controls (219.11,199.1) and (208.81,209.39) .. (196.12,209.39) -- (127.16,209.39) .. controls (114.46,209.39) and (104.17,199.1) .. (104.17,186.4) -- cycle ;
		\draw   (250.57,102.76) .. controls (250.57,98.84) and (253.74,95.67) .. (257.66,95.67) -- (278.91,95.67) .. controls (282.83,95.67) and (286,98.84) .. (286,102.76) -- (286,155.36) .. controls (286,159.27) and (282.83,162.44) .. (278.91,162.44) -- (257.66,162.44) .. controls (253.74,162.44) and (250.57,159.27) .. (250.57,155.36) -- cycle ;
		\draw  [color={rgb, 255:red, 0; green, 0; blue, 0 }  ,draw opacity=1 ] (301.6,57.43) .. controls (301.6,49.34) and (308.16,42.78) .. (316.25,42.78) -- (372.72,42.78) .. controls (380.81,42.78) and (387.37,49.34) .. (387.37,57.43) -- (387.37,101.38) .. controls (387.37,109.47) and (380.81,116.03) .. (372.72,116.03) -- (316.25,116.03) .. controls (308.16,116.03) and (301.6,109.47) .. (301.6,101.38) -- cycle ;
		\draw   (300.3,148.83) .. controls (300.3,140.94) and (306.69,134.54) .. (314.58,134.54) -- (371.79,134.54) .. controls (379.68,134.54) and (386.07,140.94) .. (386.07,148.83) -- (386.07,191.67) .. controls (386.07,199.56) and (379.68,205.95) .. (371.79,205.95) -- (314.58,205.95) .. controls (306.69,205.95) and (300.3,199.56) .. (300.3,191.67) -- cycle ;
		\draw   (328.55,49.73) .. controls (328.55,47.13) and (330.66,45.03) .. (333.26,45.03) -- (359.27,45.03) .. controls (361.87,45.03) and (363.98,47.13) .. (363.98,49.73) -- (363.98,63.85) .. controls (363.98,66.45) and (361.87,68.56) .. (359.27,68.56) -- (333.26,68.56) .. controls (330.66,68.56) and (328.55,66.45) .. (328.55,63.85) -- cycle ;
		\draw   (353.58,79.01) .. controls (353.58,76.09) and (355.95,73.72) .. (358.88,73.72) -- (374.76,73.72) .. controls (377.69,73.72) and (380.06,76.09) .. (380.06,79.01) -- (380.06,110.74) .. controls (380.06,113.66) and (377.69,116.03) .. (374.76,116.03) -- (358.88,116.03) .. controls (355.95,116.03) and (353.58,113.66) .. (353.58,110.74) -- cycle ;
		\draw   (310.7,79.01) .. controls (310.7,76.09) and (313.07,73.72) .. (316,73.72) -- (331.88,73.72) .. controls (334.81,73.72) and (337.18,76.09) .. (337.18,79.01) -- (337.18,110.74) .. controls (337.18,113.66) and (334.81,116.03) .. (331.88,116.03) -- (316,116.03) .. controls (313.07,116.03) and (310.7,113.66) .. (310.7,110.74) -- cycle ;
		\draw  [color={rgb, 255:red, 208; green, 2; blue, 27 }  ,draw opacity=1 ][line width=2.25]  (288,59) .. controls (290.27,50.79) and (366.27,36.79) .. (346.27,56.79) .. controls (326.27,76.79) and (309.12,53.68) .. (320.06,122.84) .. controls (331,192) and (266,144) .. (256.27,116.79) .. controls (246.53,89.58) and (285.73,67.21) .. (288,59) -- cycle ;
		\draw  [color={rgb, 255:red, 74; green, 144; blue, 226 }  ,draw opacity=1 ][line width=2.25]  (353.58,50.74) .. controls (368,35) and (382,44) .. (379,64) .. controls (376,84) and (353,71) .. (382,86) .. controls (411,101) and (370,126) .. (350,96) .. controls (330,66) and (339.16,66.47) .. (353.58,50.74) -- cycle ;
		\draw  [color={rgb, 255:red, 0; green, 0; blue, 0 }  ,draw opacity=1 ] (428.68,61.75) .. controls (428.68,45.54) and (441.82,32.4) .. (458.03,32.4) -- (546.09,32.4) .. controls (562.3,32.4) and (575.44,45.54) .. (575.44,61.75) -- (575.44,179.25) .. controls (575.44,195.46) and (562.3,208.6) .. (546.09,208.6) -- (458.03,208.6) .. controls (441.82,208.6) and (428.68,195.46) .. (428.68,179.25) -- cycle ;
		\draw   (432.57,100.41) .. controls (432.57,96.5) and (435.74,93.33) .. (439.66,93.33) -- (460.91,93.33) .. controls (464.83,93.33) and (468,96.5) .. (468,100.41) -- (468,153.01) .. controls (468,156.93) and (464.83,160.1) .. (460.91,160.1) -- (439.66,160.1) .. controls (435.74,160.1) and (432.57,156.93) .. (432.57,153.01) -- cycle ;
		\draw  [color={rgb, 255:red, 0; green, 0; blue, 0 }  ,draw opacity=1 ] (483.6,55.08) .. controls (483.6,46.99) and (490.16,40.43) .. (498.25,40.43) -- (554.72,40.43) .. controls (562.81,40.43) and (569.37,46.99) .. (569.37,55.08) -- (569.37,99.04) .. controls (569.37,107.13) and (562.81,113.69) .. (554.72,113.69) -- (498.25,113.69) .. controls (490.16,113.69) and (483.6,107.13) .. (483.6,99.04) -- cycle ;
		\draw   (482.3,146.48) .. controls (482.3,138.6) and (488.69,132.2) .. (496.58,132.2) -- (553.79,132.2) .. controls (561.68,132.2) and (568.07,138.6) .. (568.07,146.48) -- (568.07,189.33) .. controls (568.07,197.22) and (561.68,203.61) .. (553.79,203.61) -- (496.58,203.61) .. controls (488.69,203.61) and (482.3,197.22) .. (482.3,189.33) -- cycle ;
		\draw   (510.55,47.39) .. controls (510.55,44.79) and (512.66,42.69) .. (515.26,42.69) -- (553.29,42.69) .. controls (555.89,42.69) and (558,44.79) .. (558,47.39) -- (558,61.51) .. controls (558,64.11) and (555.89,66.21) .. (553.29,66.21) -- (515.26,66.21) .. controls (512.66,66.21) and (510.55,64.11) .. (510.55,61.51) -- cycle ;
		\draw   (535.58,76.67) .. controls (535.58,73.74) and (537.95,71.37) .. (540.88,71.37) -- (556.76,71.37) .. controls (559.69,71.37) and (562.06,73.74) .. (562.06,76.67) -- (562.06,108.39) .. controls (562.06,111.32) and (559.69,113.69) .. (556.76,113.69) -- (540.88,113.69) .. controls (537.95,113.69) and (535.58,111.32) .. (535.58,108.39) -- cycle ;
		\draw   (492.7,76.67) .. controls (492.7,73.74) and (495.07,71.37) .. (498,71.37) -- (513.88,71.37) .. controls (516.81,71.37) and (519.18,73.74) .. (519.18,76.67) -- (519.18,108.39) .. controls (519.18,111.32) and (516.81,113.69) .. (513.88,113.69) -- (498,113.69) .. controls (495.07,113.69) and (492.7,111.32) .. (492.7,108.39) -- cycle ;
		\draw  [color={rgb, 255:red, 208; green, 2; blue, 27 }  ,draw opacity=1 ][line width=2.25]  (470,56.66) .. controls (472.27,48.45) and (548.27,34.45) .. (528.27,54.45) .. controls (508.27,74.45) and (491.12,51.34) .. (502.06,120.5) .. controls (513,189.66) and (448,141.66) .. (438.27,114.45) .. controls (428.53,87.24) and (467.73,64.87) .. (470,56.66) -- cycle ;
		\draw  [color={rgb, 255:red, 65; green, 117; blue, 5 }  ,draw opacity=1 ][line width=2.25]  (267,142) .. controls (287,132) and (302,124) .. (328,145) .. controls (354,166) and (368.21,206.47) .. (358.61,232.73) .. controls (349,259) and (299,254) .. (279,224) .. controls (259,194) and (247,152) .. (267,142) -- cycle ;
		\draw  [color={rgb, 255:red, 65; green, 117; blue, 5 }  ,draw opacity=1 ][line width=2.25]  (434,145) .. controls (454,135) and (513,119) .. (524,145) .. controls (535,171) and (548,183) .. (524,205) .. controls (500,227) and (368,281) .. (348,251) .. controls (328,221) and (414,155) .. (434,145) -- cycle ;
		\draw  [color={rgb, 255:red, 245; green, 166; blue, 35 }  ,draw opacity=1 ][line width=2.25]  (156,180) .. controls (176,170) and (291,203) .. (307,219) .. controls (323,235) and (340,252) .. (336,273) .. controls (332,294) and (300,298) .. (246,273) .. controls (192,248) and (136,190) .. (156,180) -- cycle ;
		
		\draw (264,106) node [anchor=north west][inner sep=0.75pt]  [xscale=0.75,yscale=0.75]  {$X_{1}^{i}$};
		\draw (350,53) node [anchor=north west][inner sep=0.75pt]  [font=\footnotesize,xscale=0.75,yscale=0.75]  {$X_{2}^{i}$};
		\draw (226.67,22.27) node [anchor=north west][inner sep=0.75pt]  [font=\large,color={rgb, 255:red, 0; green, 0; blue, 0 }  ,opacity=1 ,xscale=0.75,yscale=0.75]  {$C_{1}^{i}$};
		\draw (267.02,36.41) node [anchor=north west][inner sep=0.75pt]  [font=\small,color={rgb, 255:red, 0; green, 0; blue, 0 }  ,opacity=1 ,xscale=0.75,yscale=0.75]  {$C_{2}^{i}$};
		\draw (79.69,28.52) node [anchor=north west][inner sep=0.75pt]  [xscale=0.75,yscale=0.75]  {$C_{1}^{s}$};
		\draw (409.1,20.74) node [anchor=north west][inner sep=0.75pt]  [xscale=0.75,yscale=0.75]  {$C_{1}^{j}$};
		\draw (440.27,102.85) node [anchor=north west][inner sep=0.75pt]  [xscale=0.75,yscale=0.75]  {$X_{1}^{j}$};
		\draw (535.38,47.69) node [anchor=north west][inner sep=0.75pt]  [font=\footnotesize,xscale=0.75,yscale=0.75]  {$X_{2}^{j}$};
		\draw (449.02,34.07) node [anchor=north west][inner sep=0.75pt]  [font=\small,color={rgb, 255:red, 0; green, 0; blue, 0 }  ,opacity=1 ,xscale=0.75,yscale=0.75]  {$C_{2}^{j}$};
		\draw (195,258.4) node [anchor=north west][inner sep=0.75pt]  [color={rgb, 255:red, 245; green, 166; blue, 35 }  ,opacity=1 ,xscale=0.75,yscale=0.75]  {$D_{_{0}}$};
		\draw (259.66,165.84) node [anchor=north west][inner sep=0.75pt]  [color={rgb, 255:red, 65; green, 117; blue, 5 }  ,opacity=1 ,xscale=0.75,yscale=0.75]  {$D_{_{1}}$};
		\draw (402.44,190.99) node [anchor=north west][inner sep=0.75pt]  [color={rgb, 255:red, 65; green, 117; blue, 5 }  ,opacity=1 ,xscale=0.75,yscale=0.75]  {$D'_{_{1}}$};
		\draw (278.66,74.84) node [anchor=north west][inner sep=0.75pt]  [color={rgb, 255:red, 208; green, 2; blue, 27 }  ,opacity=1 ,xscale=0.75,yscale=0.75]  {$D_{_{2}}$};
		\draw (360,270) node [anchor=north west][inner sep=0.75pt]  [xscale=0.9,yscale=0.75]  {$X$};
		\draw (355.58,82.41) node [anchor=north west][inner sep=0.75pt]  [color={rgb, 255:red, 74; green, 144; blue, 226 }  ,opacity=1 ,xscale=0.75,yscale=0.75]  {$D_{_{3}}$};
		\draw (459.44,77.99) node [anchor=north west][inner sep=0.75pt]  [color={rgb, 255:red, 208; green, 2; blue, 27 }  ,opacity=1 ,xscale=0.75,yscale=0.75]  {$D'_{_{2}}$};

	\end{tikzpicture}
	
	\caption{The circular figures represent the graphs that form two chains.}
	\label{fig2}
\end{figure}
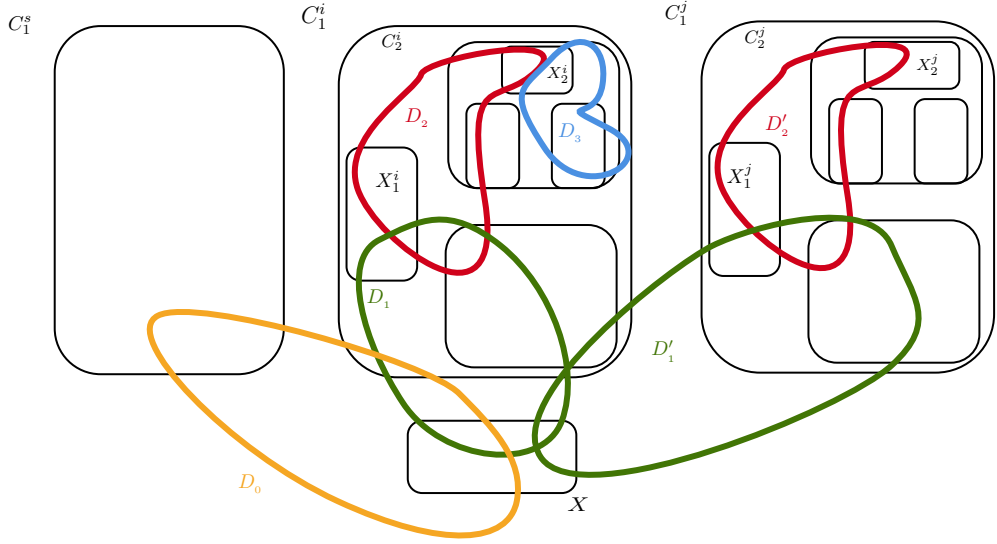



Let $\mathcal D^*$ consist of the fragments of $\mathcal D_S$, the fragment $D_{n_0}$, and all fragments that occur in the chains constructed above. Consider the auxiliary graph \(\Gamma\) defined by
$
V(\Gamma)
=
\mathcal D^*
$
and
\[
E(\Gamma)
=
\{DD':D,D'\in V(\Gamma),\ D\neq D',
\text{ and }V(D)\cap V(D')\cap S\neq\varnothing\}
\]
The purpose of \(\Gamma\) is to record how the constructions obtained
inside distinct fragments can be joined through vertices of \(S\). The
chains constructed above show that every vertex of \(\Gamma\) is
connected to \(D_{n_0}\). Hence, \(\Gamma\) is connected. Fix a spanning
tree \(\mathcal A\) of \(\Gamma\), rooted at \(D_{n_0}\), and, for every
edge \(DD'\in E(\mathcal A)\), choose a vertex
$
v_{DD'}\in V(D)\cap V(D')\cap S.
$ For each \(D\in V(\Gamma)\), put \(S_D=S\cap V(D)\). For every
\(1\leq j\leq k\), let \(H_D^j\) be the \(F\)-limit of a sequence
$
\langle T_{D,n}^j\rangle_n
$
of trees in \(D\), chosen as in Theorem~\ref{counta}, such that
$
S_D\subseteq V(H_D^j)
$
and \(H_D^1,\ldots,H_D^k\) are pairwise edge-disjoint.

For each \(n\in\mathbb N\) and \(1\leq j\leq k\), let \(U_n^j\) be the
connected component of
\[
\bigcup_{D\in V(\Gamma)}T_{D,n}^j
\]
that contains \(T_{D_{n_0},n}^j\), and let \(K_n^j\) be a spanning tree
of \(U_n^j\). Denote by \(K^j\) the \(F\)-limit of the sequence
$
\langle K_n^j\rangle_n.
$

\begin{claim}
	\label{cl1}
	\(S\subseteq V(K^j)\).
\end{claim}

\begin{proof} Let \(v\in S\). Choose
	\(D\in V(\Gamma)\) such that \(v\in S_D\). Let
	$
	D_{n_0}=D_0,D_1,\ldots,D_m=D
	$
	be the unique path in \(\mathcal A\) from \(D_{n_0}\) to \(D\), and put
	$
	v_i=v_{D_{i-1}D_i}
	$
	for every \(1\leq i\leq m\). Since \(v\in V(H_D^j)\), we have
	$
	\{n\in\mathbb N:v\in V(T_{D,n}^j)\}\in F.
	$
	Denote by \(A\) the intersection of the sets
	$
	\{n\in\mathbb N:v\in V(T_{D,n}^j)\}
	$,
	$
	\{n\in\mathbb N:v_i\in V(T_{D_{i-1},n}^j)\},
	$ and $
	\{n\in\mathbb N:v_i\in V(T_{D_i,n}^j)\},
	$
	for \(1\leq i\leq m\). Since this is a finite intersection of elements
	of \(F\), we have \(A\in F\). For every \(n\in A\), the trees
	corresponding to the fragments
	$
	D_0,D_1,\ldots,D_m
	$
	are successively joined through the vertices \(v_1,\ldots,v_m\).
	Moreover, \(v\in V(T_{D,n}^j)\). Hence, \(v\in V(U_n^j)\) and,
	since \(K_n^j\) is a spanning tree of \(U_n^j\), we also have
	\(v\in V(K_n^j)\). Therefore,
	$
	A\subseteq\{n\in\mathbb N:v\in V(K_n^j)\}.
	$
	Since \(A\in F\) and \(F\) is closed under supersets, it follows that
	$
	\{n\in\mathbb N:v\in V(K_n^j)\}\in F.
	$
	Thus, \(v\in V(K^j)\), proving that \(S\subseteq V(K^j)\).
\end{proof}
\begin{claim}
	\label{cl2}
	\(K^1,\ldots,K^k\) are pairwise
	edge-disjoint.
\end{claim}
\begin{proof}Suppose, to the contrary, that there exists an edge
	$
	e\in E(K^i)\cap E(K^j)
	$
	for some distinct \(1\leq i,j\leq k\). Let \(D\) be the unique
	fragment of \(\mathcal D\) containing \(e\). Since \(e\in E(K^i)\), it follows that
	$
	\{n\in\mathbb N:e\in E(K_n^i)\}\in F.
	$
	As \(K_n^i\) is contained in
	\(\bigcup_{D'\in V(\Gamma)}T_{D',n}^i\), and \(D\) is the unique
	fragment containing \(e\), we obtain
	$
	\{n\in\mathbb N:e\in E(T_{D,n}^i)\}\in F.
	$
	Consequently, \(e\in E(H_D^i)\). Analogously,
	\(e\in E(H_D^j)\), contradicting the fact that
	\(H_D^1,\ldots,H_D^k\) are pairwise edge-disjoint. 
\end{proof}
Therefore, by Claims~\ref{cl1} and~\ref{cl2},
\(K^1,\ldots,K^k\) are the required pairwise edge-disjoint
\(F\)-limits of trees containing \(S\).
\end{proof}

So far, we have only considered versions of Kriesell’s conjecture for infinite graphs in terms of topological $S$-Steiner trees and $F$-limits of trees which contain $S$. Our aim now is to show that if Kriesell’s conjecture holds for finite graphs, then it also holds for finitely edge-separable rayless graphs. To this end, we first require a lemma.

\begin{lemma}
	\label{lema}
	Let $G$ be a rayless finitely edge-separable graph. For every vertex $v\in V(G)$ of infinite degree, the graph $G-v$ has infinitely many components, each joined to $v$ by only finitely many edges.
\end{lemma}
\begin{proof}
	Let $v\in V(G)$ be a vertex of infinite degree. Suppose that there exists a connected component $C$ of $G-v$ which is joined to $v$ by infinitely many edges, that is, \(v\) has infinitely many neighbors in \(C\). Let $A\subseteq V(C)$ be an infinite set of neighbors of $v$ in $C$. Since $G$ is rayless, it follows that $C$ is also rayless. Then there exists a subdivided star in $C$ with leaves in $A$ and center $u\in C$. Notice that the vertices $u$ and $v$ are infinitely edge-connected, which contradicts $G$ being finitely edge-separable.
	
	Suppose that $G - v$ has only finitely many connected components. Since $v$  has infinite degree, one of these components must be joined to $v$ by infinitely many edges. This contradicts the preceding argument.
\end{proof}

\begin{theorem}
	\label{quase}
	Let $G$ be a countable, connected, finitely edge-separable, rayless graph, and let $S\subseteq V(G)$ be a $2k$-edge-connected subset. If Conjecture \ref{kriesell} holds, then $G$ contains $k$ pairwise edge-disjoint $S$-Steiner trees.
\end{theorem}

\begin{proof}
Let $\widetilde{G}$ be the expanded graph of $G$. For each \(v\in B\), let
$
\langle C_n^v\rangle_n
$
be an enumeration of the connected components of \(G-v\) that are
adjacent to \(v\). By Lemma~\ref{lema}, the set \(E(v,C_n^v)\) is
finite for every \(n\in\mathbb{N}\). Since \(v\) has infinite degree,
there are infinitely many such components.

For every \(n\in\mathbb{N}\), set
$
d_n^v=|E(v,C_n^v)|,$ $p_0^v=0$ and $
p_{n+1}^v=p_n^v+d_n^v.
$
We choose the enumeration
$
E(v)=\{e_i^v:i\in\mathbb{N}\}
$
in such a way that
$
E(v,C_n^v)
=
\{e_i^v:p_n^v\leq i<p_{n+1}^v\}
$
for every \(n\in\mathbb{N}\). Thus, the edges joining \(v\) to the
same connected component of \(G-v\) receive consecutive indices, and
the blocks corresponding to the components \(C_n^v\) occur in
increasing order along \(r_v\). This requirement will be essential in the proof of
Claim~\ref{cray2}. Without this requirement, \(\widetilde{G}\) may contain
an end of degree at least \(2\). Indeed, let \(G\) consist of a vertex
\(u\) joined to every vertex of infinitely many pairwise disjoint copies
of \(K_4\). A suitable interlaced enumeration of the edges incident with
\(u\) produces two vertex-disjoint equivalent rays in \(\widetilde{G}\),
as illustrated in Figure~\ref{contra}.

\begin{figure}[htbp]
	\centering
	
	\tikzset{every picture/.style={line width=0.4pt}}
	
	\begin{tikzpicture}[
		x=0.6pt,
		y=0.6pt,
		yscale=-1,
		xscale=1
		]
		
		\foreach \i/\x in {3/115,4/180,5/245,6/310,7/375,8/440,9/505,10/570,11/635,12/700}{
			\coordinate (c\i) at (\x,190);
		}
		
		\foreach \i/\xa/\xb in {3/140/155,4/205/220,5/270/285,6/335/350,7/400/415,8/465/480,9/530/545,10/595/610,11/660/675}{
			\coordinate (w\i a) at (\xa,169);
			\coordinate (w\i b) at (\xb,169);
			\coordinate (w\i c) at (\xa,211);
			\coordinate (w\i d) at (\xb,211);
		}
		
		\foreach \i/\left/\right in {3/c3/c4,4/c4/c5,5/c5/c6,6/c6/c7,7/c7/c8,8/c8/c9,9/c9/c10,10/c10/c11,11/c11/c12}{
			\draw[black,line width=0.4pt] (\left)--(w\i a);
			\draw[black,line width=0.4pt] (\left)--(w\i b);
			\draw[black,line width=0.4pt] (\left)--(w\i c);
			\draw[black,line width=0.4pt] (\left)--(w\i d);
			
			\draw[black,line width=0.4pt] (\right)--(w\i a);
			\draw[black,line width=0.4pt] (\right)--(w\i b);
			\draw[black,line width=0.4pt] (\right)--(w\i c);
			\draw[black,line width=0.4pt] (\right)--(w\i d);
			
			\draw[black,line width=0.4pt] (w\i a)--(w\i b);
			\draw[black,line width=0.4pt] (w\i a)--(w\i c);
			\draw[black,line width=0.4pt] (w\i a)--(w\i d);
			\draw[black,line width=0.4pt] (w\i b)--(w\i c);
			\draw[black,line width=0.4pt] (w\i b)--(w\i d);
			\draw[black,line width=0.4pt] (w\i c)--(w\i d);
		}
		
		\coordinate (q0a) at (170,55);
		\coordinate (q0b) at (300,55);
		\coordinate (q0c) at (210,105);
		\coordinate (q0d) at (340,105);
		
		\coordinate (q1a) at (430,55);
		\coordinate (q1b) at (560,55);
		\coordinate (q1c) at (470,105);
		\coordinate (q1d) at (600,105);
		
		\draw[black,line width=0.4pt] (q0a)--(q0b);
		\draw[black,line width=0.4pt] (q0c)--(q0d);
		\draw[black,line width=0.4pt] (q0a)--(q0c);
		\draw[black,line width=0.4pt] (q0a)--(q0d);
		\draw[black,line width=0.4pt] (q0b)--(q0c);
		\draw[black,line width=0.4pt] (q0b)--(q0d);
		
		\draw[black,line width=0.4pt] (q1a)--(q1b);
		\draw[black,line width=0.4pt] (q1c)--(q1d);
		\draw[black,line width=0.4pt] (q1a)--(q1c);
		\draw[black,line width=0.4pt] (q1a)--(q1d);
		\draw[black,line width=0.4pt] (q1b)--(q1c);
		\draw[black,line width=0.4pt] (q1b)--(q1d);
		
		\draw[black,line width=0.4pt] (c3)--(q0a);
		\draw[black,line width=0.4pt] (q0b)--(c6);
		\draw[black,line width=0.4pt] (c5)--(q0c);
		
		\draw[black,line width=0.4pt]
		(q0d) .. controls (390,110) and (455,145) .. (c8);
		
		\draw[black,line width=0.4pt] (c7)--(q1a);
		\draw[black,line width=0.4pt] (q1b)--(c10);
		\draw[black,line width=0.4pt] (c9)--(q1c);
		
		\draw[black,line width=0.4pt]
		(q1d) .. controls (650,110) and (715,145) .. (c12);
		
		\draw[
		color={rgb,255:red,74;green,144;blue,226},
		line width=1.2pt
		]
		(c3)--(q0a)--(q0b)--(c6);
		
		\draw[
		color={rgb,255:red,74;green,144;blue,226},
		line width=1.2pt
		]
		(c6)--(w6a)--(c7);
		
		\draw[
		color={rgb,255:red,74;green,144;blue,226},
		line width=1.2pt
		]
		(c7)--(q1a)--(q1b)--(c10);
		
		\draw[
		color={rgb,255:red,74;green,144;blue,226},
		line width=1.2pt
		]
		(c10)--(w10a)--(c11);
		
		\draw[
		color={rgb,255:red,74;green,144;blue,226},
		line width=1.2pt,
		dashed
		]
		(c11) .. controls (665,135) and (690,105) .. (720,80);
		
		\draw[
		color={rgb,255:red,208;green,2;blue,27},
		line width=1.2pt
		]
		(c4)--(w4c)--(c5);
		
		\draw[
		color={rgb,255:red,208;green,2;blue,27},
		line width=1.2pt
		]
		(c5)--(q0c)--(q0d);
		
		\draw[
		color={rgb,255:red,208;green,2;blue,27},
		line width=1.2pt
		]
		(q0d) .. controls (390,110) and (455,145) .. (c8);
		
		\draw[
		color={rgb,255:red,208;green,2;blue,27},
		line width=1.2pt
		]
		(c8)--(w8c)--(c9);
		
		\draw[
		color={rgb,255:red,208;green,2;blue,27},
		line width=1.2pt
		]
		(c9)--(q1c)--(q1d);
		
		\draw[
		color={rgb,255:red,208;green,2;blue,27},
		line width=1.2pt
		]
		(q1d) .. controls (650,110) and (715,145) .. (c12);
		
		\draw[
		color={rgb,255:red,208;green,2;blue,27},
		line width=1.2pt,
		dashed
		]
		(c12) .. controls (720,215) and (730,245) .. (750,270);
		
		\foreach \i in {3,4,5,6,7,8,9,10,11,12}{
			\fill[black] (c\i) circle (2.5pt);
		}
		
		\foreach \i in {3,4,5,6,7,8,9,10,11}{
			\foreach \letter in {a,b,c,d}{
				\fill[black] (w\i\letter) circle (1.8pt);
			}
		}
		
		\foreach \x in {q0a,q0b,q0c,q0d,q1a,q1b,q1c,q1d}{
			\fill[black] (\x) circle (2.5pt);
		}
		
		\foreach \x in {c3,c6,c7,c10,c11,w6a,w10a,q0a,q0b,q1a,q1b}{
			\fill[
			color={rgb,255:red,74;green,144;blue,226}
			]
			(\x) circle (2.6pt);
		}
		
		\foreach \x in {c4,c5,c8,c9,c12,w4c,w8c,q0c,q0d,q1c,q1d}{
			\fill[
			color={rgb,255:red,208;green,2;blue,27}
			]
			(\x) circle (2.6pt);
		}
		
		\node[
		font=\small,
		color={rgb,255:red,74;green,144;blue,226}
		] at (235,65) {\(r\)};
		
		\node[
		font=\small,
		color={rgb,255:red,208;green,2;blue,27}
		] at (275,122) {\(r'\)};
		
	\end{tikzpicture}
	\vspace{-1cm}
	\caption{An end of degree two in an expanded graph.}
	\label{contra}
\end{figure}
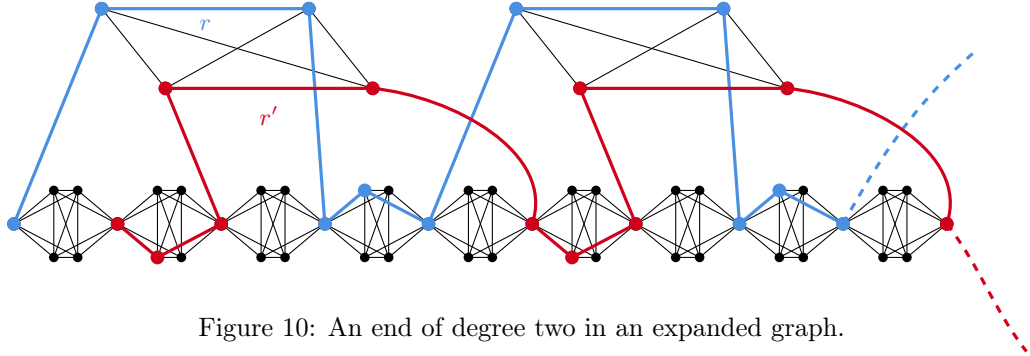
	
\begin{claim}
	Every ray \(r\in\mathcal{R}(\widetilde{G})\) is equivalent to
	\(r_u\) for a unique vertex \(u\in V(G)\) of infinite degree.
	\label{cray}
\end{claim}

\begin{proof}
We first prove existence. Let \(r\) be a ray of \(\widetilde{G}\) and
suppose, towards a contradiction, that \(V(r)\cap V(r_u)\) is finite
for every vertex \(u\) of infinite degree. Write
\(r=(z_n)_{n\in\mathbb N}\), where the \(z_n\) are its vertices and the \(z_iz_{i+1}\) are its edges. Consider the subgraph
\(H=\rho^{-1}[r]\) of \(G\), the projection of \(r\) to \(G\) under the inverse edge correspondence \(\rho^{-1}\).

Let \(I\) be the set of all \(i\in\mathbb{N}\) such that
\(z_iz_{i+1}\in\rho[E(G)]\). We first show that \(I\) is infinite.
Suppose that \(I\) is finite. Choose \(N\in\mathbb{N}\) such that
\(i<N\) for every \(i\in I\). Then the tail
\((z_n)_{n\geq N}\) contains no edge of
\(\rho[E(G)]\).
The only edges of \(\widetilde{G}\) that do not belong to
\(\rho[E(G)]\) are edges belonging to some \(r_v\), where \(v\) has
infinite degree. Moreover, distinct \(2k\)-rays \(r_u\) and \(r_v\)
are connected only by edges belonging to \(\rho[E(G)]\). Therefore,
the tail \(( z_n)_{n\geq N}\) is contained in \(r_v\) for
some vertex \(v\) of infinite degree.
This contradicts the assumption that \(V(r)\cap V(r_v)\) is finite. Therefore \(I\) is infinite, and hence \(E(H)\) is infinite.

Enumerate \(I\) in increasing order as
\( (i_j)_j\). For every \(j\in\mathbb{N}\), let \(e_j\)
be the unique edge of \(G\) such that
$
\rho(e_j)=z_{i_j}z_{\,i_j+1}.
$
Here \(i_j+1\) is the index of the vertex immediately after \(z_{i_j}\) on \(r\), whereas \(i_{j+1}\) is the next member of \(I\). We claim that \(e_j\) and \(e_{j+1}\)
have a common endpoint.
Indeed, by the choice of \(i_j\) and \(i_{j+1}\), the path of \(r\)
from \(z_{i_j+1}\) to \(z_{i_{j+1}}\) contains no edge of
\(\rho[E(G)]\). Hence this path is contained either in some
\(r_{v_j}\), where \(v_j\) has infinite degree, or consists of a
single vertex \(v_j\) of finite degree. Consequently, the endpoint
\(z_{i_j+1}\) of \(\rho(e_j)\) and the endpoint
\(z_{i_{j+1}}\) of \(\rho(e_{j+1})\) both correspond to \(v_j\).
Therefore, \(v_j\) is a common endpoint of \(e_j\) and \(e_{j+1}\).
The edges \(e_j\) are pairwise distinct because \(r\) does not repeat
edges and \(\rho\) is injective. Thus,
\((e_j)_j\) is an infinite trail in \(G\). In
particular, the subgraph \(H=\rho^{-1}[r]\) is connected.

We now show that \(H\) is locally finite. If \(v\) has finite
degree in \(G\), then \(v\) clearly has finite degree in \(H\).
Suppose that \(v\) has infinite degree in \(G\). Every edge of
\(H\) incident with \(v\) corresponds, under \(\rho\), to an edge
of \(r\) incident with some vertex \((v,i)\). Moreover, distinct
edges of \(G\) incident with \(v\) correspond to edges incident
with distinct vertices \((v,i)\). Since every vertex \((v,i)\)
belongs to \(r_v\), the assumption that
\(V(r)\cap V(r_v)\) is finite implies that \(v\) has finite
degree in \(H\).

Thus, \(H\) is an infinite, connected and locally finite subgraph
of \(G\). Then \(H\) contains a ray,
contradicting the fact that \(G\) is rayless.

Therefore, there exists a vertex \(u\) of infinite degree such
that \(V(r)\cap V(r_u)\) is infinite. Then \(r\) is
equivalent to \(r_u\). Uniqueness follows from the fact that if $r_u$ and $r_v$ are equivalent rays in $\widetilde{G}$, then $u$ and $v$ are vertices infinitely edge-connected in $G$.
\end{proof}

	\begin{claim} Every end of $\widetilde{G}$ has degree one.
		\label{cray2}
	\end{claim}
\begin{proof}
	Suppose, towards a contradiction, that an end
	\([s]\in\Omega(\widetilde{G})\) has degree at least two.
	By Claim~\ref{cray}, there exists a vertex \(u\in V(G)\) of
	infinite degree such that \([s]=[r_u]\). Therefore, there exist
	two vertex-disjoint rays \(r,r'\in[r_u]\).
	
	Let \(\langle C_n^u\rangle_n\) be the enumeration of the connected
	components of \(G-u\) used in the construction of
	\(\widetilde{G}\). For simplicity, write \(p_n=p_n^u\). By the
	choice of the enumeration of \(E(u)\), we have
	\(E(u,C_n^u)=\{e_i^u:p_n\leq i<p_{n+1}\}\). Consequently, the
	vertices of \(r_u\) through which the expanded
	copy \(\widetilde{C}_n^u\) is attached are 
	\(A_n=N_{\widetilde{G}}(\widetilde{C}_n^u)\cap V(r_u)
	=\{(u,i):p_n\leq i<p_{n+1}\}\). In particular, each \(A_n\) is
	finite, and the sets \(\langle A_n\rangle_n\) occur consecutively
	and without interlacing in \(r_u\).
	
	For each \(n\in\mathbb{N}\), let \(x_n=(u,p_{n+1})\). Let \(L_n\)
	be the component of \(\widetilde{G}-x_n\) containing \((u,0)\),
	and let \(K_n\) be the component containing a tail of \(r_u\).
Notice that \(L_n\neq K_n\).
Indeed, all the edges joining \(u\) to
	\(C_0^u,\ldots,C_n^u\) are represented at vertices \((u,i)\)
	with \(i<p_{n+1}\), whereas all the edges joining \(u\) to
	\(C_m^u\), for \(m>n\), are represented at vertices \((u,i)\)
	with \(i\geq p_{n+1}\). Moreover, distinct components of \(G-u\)
	have no edges between them. Finally, in  \(r_u\),
	every path from a vertex \((u,i)\), with \(i<p_{n+1}\), to a
	vertex \((u,j)\), with \(j>p_{n+1}\), must pass through \(x_n\).
	For each \(i\in\mathbb{N}\), denote by
	\(W_i^u=\{w_j^u:2ki\leq j<2k(i+1)\}\) the vertex set of the copy of
	\(K_{2k}\) placed between \((u,i)\) and \((u,i+1)\). By construction,
	every vertex of \(W_{p_{n+1}-1}^u\) is adjacent, outside this copy of
	\(K_{2k}\), only to \((u,p_{n+1}-1)\) and
	\(x_n=(u,p_{n+1})\). Similarly, every vertex of
	\(W_{p_{n+1}}^u\) is adjacent, outside its copy of \(K_{2k}\), only to
	\(x_n\) and \((u,p_{n+1}+1)\). Moreover, there are no edges between
	\(W_{p_{n+1}-1}^u\) and \(W_{p_{n+1}}^u\). Hence, after removing
	\(x_n\), there is no path in \(r_u\) from a vertex
	\((u,i)\), with \(i<p_{n+1}\), to a vertex \((u,j)\), with
	\(j>p_{n+1}\). Therefore, \(x_n\) separates the part containing the
	blocks corresponding to \(C_0^u,\ldots,C_n^u\) from a tail of \(r_u\).
	
	Let \(v_0\) and \(v_0'\) be the initial vertices of \(r\) and
	\(r'\), respectively. We may choose \(N\in\mathbb{N}\) such that
	\(v_0,v_0'\in L_N\). To see this, consider each of these vertices
	separately. If it belongs to \(\widetilde{C}_j^u\), choose
	\(N>j\). If it is of the form \((u,i)\), choose \(N\) such that
	\(p_{N+1}>i\). Finally, if it is a vertex \(w_\ell^u\) in the copy
	of \(K_{2k}\) between \((u,i)\) and \((u,i+1)\), choose \(N\)
	such that \(p_{N+1}>i+1\). Since only \(v_0\) and \(v_0'\) are
	being considered, the same index \(N\) can be chosen for both.
	
	Since \([r]=[r']=[r_u]\), the definition of equivalent rays
	implies that both \(r\) and \(r'\) have tails contained in
	\(K_N\). However, their initial vertices \(v_0\) and \(v_0'\)
	belong to \(L_N\), and \(L_N\neq K_N\). Therefore, each of the
	rays \(r\) and \(r'\) must pass through \(x_N\). Hence
	\(x_N\in V(r)\cap V(r')\), contradicting the fact that \(r\) and
	\(r'\) are vertex-disjoint. Thus, no end of \(\widetilde{G}\) has degree at least two.
\end{proof}

	By Claim~\ref{cray2}, every topological tree of $\widetilde{G}$ is a tree of $\widetilde{G}$. Consider $k$ pairwise edge-disjoint topological $\widetilde{S}$-Steiner trees of $\widetilde{G}$ given by Theorem \ref{localfin}.  The corresponding subgraphs in $G$ are pairwise edge-disjoint, connected and contain $S$. Taking a spanning tree of each subgraph gives us $k$ pairwise edge-disjoint $S$-Steiner trees. 
\end{proof}

	\begin{theorem}
		\label{ray}
		Let $G$ be a connected, finitely edge-separable, rayless graph and $S\subseteq V(G)$ be a $2k$-edge-connected subset. If Conjecture \ref{kriesell} holds, then $G$ contains $k$ pairwise edge-disjoint $S$-Steiner trees.
	\end{theorem}
	\begin{proof}
		This follows from Theorem~\ref{quase} and the proof of Theorem \ref{final}.
\end{proof}

\begin{corollary}
	\label{final1}
	Let $G$ be a $2k$-edge-connected, rayless and finitely edge-separable graph. Then $G$ contains $k$ pairwise edge-disjoint spanning trees.
\end{corollary}

\begin{proof}
	It follows directly from Theorems~\ref{nash} and~\ref{ray} for $S=V(G)$.
\end{proof}

We conclude by raising the following questions.

\begin{question}
Does the version of Kriesell's conjecture with topological Steiner trees or $F$-limits of trees hold for finitely edge-separable uncountable graphs?
\end{question}

\begin{question}
	
	Does Theorem~\ref{ray} hold for rayless graphs that are not finitely edge-separable?
	\end{question}

\section*{Acknowledgments}
	
The first author acknowledges the support of the Fundação de Amparo à Pesquisa do Estado de São Paulo (FAPESP) through grant number 2025/12199-3. The third author acknowledges the support of the Coordenação de Aperfeiçoamento de Pessoal de Nível Superior - Brasil (CAPES) - Finance Code 001.
	
\bibliographystyle{plain}

\bibliography{refstrees}

\begin{thebibliography}{10}

\bibitem{AharoniThomassen1989}
Ron Aharoni and Carsten Thomassen.
\newblock Infinite, highly connected digraphs with no two arc-disjoint spanning
  trees.
\newblock {\em Journal of Graph Theory}, 13(1):71--74, 1989.

\bibitem{guilherme}
Leandro Aurichi, Paulo~Magalhães Júnior, and Guilherme Pinto.
\newblock On orientations preserving edge-connectivity in infinite graphs.
\newblock {\em Journal of Logic and Computation}, 36(5):exag028, July 2026.

\bibitem{Aurichi2026}
Leandro Aurichi, Paulo~Magalhães Júnior, and Luisa Seixas.
\newblock Limits of cycles and cover conjectures.
\newblock {\em Discrete Mathematics}, 349:114724, 2026.

\bibitem{aurichi2025cyclecoversinfinitebipartite}
Leandro Aurichi, Paulo~Magalhães Júnior, and Lyubomyr Zdomskyy.
\newblock On cycle covers of infinite bipartite graphs.
\newblock \url{https://arxiv.org/abs/2504.02816}, 2025.

\bibitem{boska2025edgedirectioncompactedgeendspaces}
Gustavo Boska, Matheus Duzi, and Paulo~Magalhães Júnior.
\newblock On edge-direction and compact edge-end spaces.
\newblock \url{https://arxiv.org/abs/2503.19088}, 2025.

\bibitem{DeVosMcDonaldPivotto2016}
Matt DeVos, Jessica McDonald, and Irene Pivotto.
\newblock Packing {S}teiner trees.
\newblock {\em Journal of Combinatorial Theory, Series B}, 119:178--213, 2016.

\bibitem{Diestel2025}
Reinhard Diestel.
\newblock {\em Graph Theory}, volume 173 of {\em Graduate Texts in
  Mathematics}.
\newblock Springer-Verlag, Heidelberg, 6th edition, 2025.

\bibitem{frank2003}
András Frank, Tamás Király, and Matthias Kriesell.
\newblock On decomposing a hypergraph into $k$ connected sub-hypergraphs.
\newblock {\em Discrete Applied Mathematics}, 131(2):373--383, 2003.

\bibitem{HahnLavioletteSiran1997}
Geňa Hahn, François Laviolette, and Jozef Širáň.
\newblock Edge-ends in countable graphs.
\newblock {\em Journal of Combinatorial Theory, Series B}, 70(2):225--244,
  1997.

\bibitem{Kriesell2003}
Matthias Kriesell.
\newblock Edge-disjoint trees containing some given vertices in a graph.
\newblock {\em Journal of Combinatorial Theory, Series B}, 88:53--65, 2003.

\bibitem{Laviolette2005}
François Laviolette.
\newblock Decompositions of infinite graphs: I — bond-faithful
  decompositions.
\newblock {\em Journal of Combinatorial Theory, Series B}, 94(2):259--277,
  2005.

\bibitem{Li2018}
Hui Li, Baoyindureng Wu, Jiajia Meng, and Ying Ma.
\newblock Steiner tree packing number and tree connectivity.
\newblock {\em Discrete Mathematics}, 341:1945--1951, 2018.

\bibitem{nash1961edge}
Crispin Nash-Williams.
\newblock Edge-disjoint spanning trees of finite graphs.
\newblock {\em Journal of the London Mathematical Society}, 36(1):445--450,
  1961.

\bibitem{pitz2025metrizationtheoremedgeendspaces}
Max Pitz.
\newblock A metrization theorem for edge-end spaces of infinite graphs.
\newblock {\em Proceedings of the American Mathematical Society},
  154(5):2209--2219, 2026.

\bibitem{real2025subbasepropertydescribingedgeend}
Lucas Real.
\newblock A subbase property for describing edge-end spaces.
\newblock \url{https://arxiv.org/abs/2508.17424}, 2025.

\bibitem{Schmidt1982}
Rudolf Schmidt.
\newblock {\em Ein Reduktionsverfahren für weg-endliche Graphen}.
\newblock PhD thesis, Universität Hamburg, Hamburg, 1982.

\bibitem{tutte1961problem}
William Tutte.
\newblock On the problem of decomposing a graph into $n$ connected factors.
\newblock {\em Journal of the London Mathematical Society}, 36:221--230, 1961.

\end{thebibliography}

\Addresses
	
\end{document}